\documentclass[a4paper,12pt,oneside]{article}
\usepackage[english]{babel}
\usepackage[T1]{fontenc} 
\usepackage[utf8]{inputenc}
\usepackage{amsthm}
\usepackage{amsmath}
\usepackage{amssymb}  
\usepackage{indentfirst}
\usepackage{fancyhdr}
\usepackage{amsthm}
\usepackage{graphicx}
\usepackage{pdfpages}
\usepackage{esint}
\usepackage{url}

\definecolor{dg}{rgb}{0.01, 0.75, 0.24}

\numberwithin{equation}{section}

\theoremstyle{plain} 
\newtheorem{thm}{Theorem}[section]

\newtheorem{lem}[thm]{Lemma} 
\newtheorem{prop}[thm]{Proposition} 
\newtheorem{rmk}[thm]{Remark}

\def\io{\int_{\Omega}}
\def\ibtt{\int_{B_{t'}}}
\def\ibt {\int_{B_{t}}}
\def\ibR {\int_{B_{R}}}

\def\ibr2 {\int_{B_{\frac{R}{2}}}}
\def\ibro2 {\int_{B_{\frac{\rho}{2}}}}
\def\ibs{\int_{B_{s}}}
\def\iB{\int_B}
\def\fhi{\varphi}

\def\dd{\textrm{d}}

\def\fibR {\fint_{B_{R}}}
\def\fibrho {\fint_{B_\rho}}
\def\fibr2 {\displaystyle\fint_{B_{\frac{R}{2}}}}
\def\fibro2 {\fint_{B_{\frac{\rho}{2}}}}

\usepackage{geometry}
\allowdisplaybreaks[4]

\begin{document}

\title{Higher differentiability results\\ in the scale of Besov spaces \\
to a class of double-phase obstacle problems}
\author {   
\sc{Antonio Giuseppe Grimaldi}\thanks{Dipartimento di Matematica e Applicazioni "R. Caccioppoli", Università degli Studi di Napoli "Federico II", Via Cintia, 80126 Napoli
 (Italy). E-mail: \textit{antoniogiuseppe.grimaldi@unina.it}} 
  \;\textrm{and}   Erica Ipocoana \thanks{University of Modena and Reggio Emilia, Dipartimento di Scienze Fisiche, Informatiche e Matematiche, via Campi 213/b, I-41125 Modena (Italy).
E-mail: \textit{erica.ipocoana@unipr.it}}}

\maketitle
\maketitle
\begin{abstract}
We study the higher fractional differentiability properties of the gradient of the solutions to variational obstacle problems of the form
\begin{gather*}
\min \biggl\{ \int_{\Omega} F(x,w,Dw) d x \ : \ w \in \mathcal{K}_{\psi}(\Omega)  \biggr\},
\end{gather*}
with $F$ double phase functional of the form
\begin{equation*}
F(x,w,z)=b(x,w)(|z|^p+a(x)|z|^q),
\end{equation*}
where $\Omega$ is a bounded open subset of $\mathbb{R}^n$, $\psi \in W^{1,p}(\Omega)$ is a fixed function called \textit{obstacle} and $\mathcal{K}_{\psi}(\Omega)= \{ w \in W^{1,p}(\Omega) : w \geq \psi \ \text{a.e. in} \ \Omega  \}$ is the class of admissible functions. Assuming that the gradient of the obstacle belongs to a suitable Besov space, we are able to prove that the gradient of the solution preserves some fractional differentiability property.
\end{abstract}


\section{Introduction}

In this paper we study the higher fractional differentiability properties of the gradient of the solutions $u \in W^{1,p}(\Omega)$ to variational obstacle problems of the form
\begin{gather}\label{obpro}
\min \biggl\{ \displaystyle\int_{\Omega} F(x,w,Dw) \dd x \ : \ w \in \mathcal{K}_{\psi}(\Omega)  \biggr\},
\end{gather}
where the energy density $F: \Omega \times \mathbb{R} \times \mathbb{R}^n \rightarrow \mathbb{R}$ is defined by 
\begin{equation}\label{integrand}
F(x,w,z)=b(x,w)H(x,z),  
\end{equation}
being
\begin{equation}\label{integrandh}
H(x,z) = |z|^p+ a(x)|z|^q,
\end{equation}
where $2 \leq p < q$.\\
Here $\Omega$ is a bounded open set of $\mathbb{R}^n$, $n \geq 2$, the function $\psi: \Omega \rightarrow [-\infty, + \infty)$, called \textit{obstacle}, belongs to the Sobolev class $W^{1,p}(\Omega)$ and the class $\mathcal{K}_{\psi}(\Omega) $ is defined as follows
\begin{gather}
\mathcal{K}_{\psi}(\Omega)= \{ w \in W^{1,p}(\Omega) : w \geq \psi \ \text{a.e. in} \ \Omega  \}. \notag
\end{gather}
Note that the set $\mathcal{K}_{\psi}(\Omega) $ is not empty since $\psi \in \mathcal{K}_{\psi}(\Omega) $.
\\We assume that the coefficients $a(x)$ and $b(x,w)$ satisfy the following assumptions:

\textbf{Assumption 1}

\begin{itemize}
\item[(i)] $a : \Omega \rightarrow [0, +\infty)$ is a bounded and measurable function such that $$|a(x)-a(y)| \leq \omega_a(|x-y|),$$ for all $x,y \in \Omega$, where $\omega_a : \mathbb{R}^+ \rightarrow [0,1]$ is defined by $\omega_a(\rho)= \min \{ \rho^\alpha, 1 \}$, for some $\alpha \in (0,1)$;
\item[(ii)] the function $b: \Omega \times \mathbb{R} \rightarrow (0,+\infty)$ is a bounded Carathéodory function, i.e. there exist $0< \nu \leq L$ such that
$$0 < \nu \leq b(x,w) \leq L < \infty.$$
\end{itemize}

\textbf{Assumption 2}

\begin{itemize}
\item[(i)] there exists a function $\omega_b : \mathbb{R}^+ \rightarrow [0,1]$  defined by $\omega_b(\rho)= \min \{ \rho^\beta, 1 \}$, for some $\beta \in (0,1)$, such that
$$|b(x,u)-b(y,v)| \leq \omega_b(|x-y|+|u-v|),$$ for all $x,y \in \Omega$ and every $u,v \in \mathbb{R}$.
\end{itemize}
The energy density given by \eqref{integrand} is a model case of functions $F$ satisfying the following set of conditions
$$ \nu_1 |z|^p \leq F(x,w,z) \leq L_1 (1+|z|^q) \eqno{\rm{{ (F1)}}}$$
$$ \nu_2 |z|^{p-2} |\lambda|^2 \leq \langle D_{zz}F(x,w,z) \lambda, \lambda \rangle \leq L_2 (1+|z|^{q-2}) |\lambda|^2  \eqno{\rm{{ (F2)}}}$$
$$ |F(x_1,w_1,z)-F(x_2,w_2,z)| \leq l_1 \omega_\delta(|x_1-x_2|+|w_1-w_2|)(1+|z|^q)  \eqno{\rm{{ (F3)}}}$$
for all $x,x_1,x_2 \in \Omega$, $w,w_1,w_2 \in \mathbb{R}$ and every $z,\lambda \in \mathbb{R}^n$, where $0 <\nu_1 \leq L_1$, $0 <\nu_2 \leq L_2$, $l_1 \geq 1$ are fixed constants and $\omega_\delta : \mathbb{R}^+ \rightarrow [0,1]$ is a function defined by $\omega_\delta(\rho)= \min \{ \rho^\delta, 1 \}$, for some $\delta \in (0,1)$ depending on $\alpha$ and $\beta$ introduced in Assumption 1 and 2 respectively.
We point out that the choice of stating Assumption 1 and 2 separately is due to the fact that they are needed independently.

The obstacle problem appeared in the mathematical literature in the work of Stampacchia \cite{stampacchia} in the special case $\psi = \chi_E$ and related to the capacity of a subset $E \Subset \Omega$; in an earlier independent work, Fichera \cite{fichera} solved the first unilateral problem, the so-called \textit{Signorini problem} in elastostatics.

It is usually observed that the regularity of solutions to the obstacle problems is influenced by the one of the obstacle; for example, for linear obstacle problems, obstacle and solutions have the same regularity \cite{brezis.kinderlehrer,caffarelli.kinderlehrer,kinderlehrer.stampacchia}. This does not apply in the nonlinear setting, hence along the years, there have been intense research activities for the regularity of the obstacle problem in this direction.
In the case of standard growth conditions, Eleuteri and Passarelli di Napoli \cite{eleuteri.passarelli2018} proved that an extra differentiability of integer or fractional order of the gradient of the obstacle transfers to the gradient of the solutions, provided the partial map $x \mapsto D_\xi\tilde{F}(x,\xi)$ possesses a suitable differentiability property, where $\tilde{F}$ is a general integrand independent of the $w-$variable.  
\\Recently, it was proved in \cite{gavioli1,gavioli2} that the weak differentiability of integer order of the partial map $x \mapsto D_\xi\tilde{F}(x, \xi)$ is a sufficient condition to prove that an extra differentiability of integer order of the gradient of the obstacle transfers to the gradient of the solutions to obstacle problems with $p,q$-growth conditions. 
This property was generalized also for fractional differentiability, connected to Besov spaces in   
 \cite{grimaldi.ipocoana}.\\
It is worth noticing that double phase functionals
are a useful tool to study the behaviour of
strongly anisotropic materials whose hardening properties are strongly dependent on the point and connected to the exponent
ruling the growth of the gradient variable. The
coefficient $a(\cdot)$ regulates the mixture between two different materials, with
$p$ and $q$ hardening, respectively (see, for instance, \cite{zhikov1,zhikov2}).
The regularity properties of local minimizers to such functionals
recently have been investigated for unconstrained problems. In particular, we quote the work \cite{colombo.mingione} by Colombo and Mingione where the functional $H(x,Du)$  has been considered (see \eqref{integrandh}), \cite{baroni.colombo.mingione} by Baroni, Colombo and Mingione who studied the integrand defined in \eqref{integrand} and \cite{coscia} by Coscia, who dealt with the functional defined by
$$ \mathcal{F} (w,\Omega):= \displaystyle\int_{\Omega} b(x,w)[|Dw|^p+a(x)|Dw|^p \text{log}(e+|Dw|)]\dd x.$$
Furthermore, a higher fractional differentiability has been proved for solutions to double phase elliptic obstacle problems in \cite{zhang.zheng}.  We also recall that when referring to $p,q$-growth conditions, in order to ensure the regularity of minima, a smallness condition on the gap $q/p>1$ is necessary (see, for instance, the counterexamples in \cite{fonseca.maly.mingione,giaquinta,marcellini1991}).\\

The main difficulty of this work is the dependence of our double phase functional both on the $x-$variable and the $w-$variable, where the map $w\mapsto b(x,w)H(x,z)$ is non-differentiable. In order to deal with this issue, we follow the strategy  proposed  in \cite{kristensen.mingione} and later used in \cite{eleuteri.passarelli2020}.  Namely, we introduce the so-called "freezed" functional defined in \eqref{froint} and the solution to the corresponding obstacle problem (see \eqref{frozen}) for which we prove a higher differentiability result in the scale of Besov spaces following the argument in \cite{grimaldi.ipocoana}. The idea is to compare the solution $u$ to the original obstacle problem \eqref{obpro} and the solution $v$ to the "freezed" one \eqref{frozen}. More precisely, we estimate the fractional difference quotients of $u$ and $v$, in an integral sense, gaining a Besov regularity for $u$. In order to do so, we also have to derive some ad hoc higher integrability results, both at the interior and up to boundary, that is for the solution $u$ of the original obstacle problem \eqref{obpro} and the solution $v$ to the freezed one \eqref{frozen} respectively. The first one is obtained adapting the argument in \cite{eleuteri2004}, while the second one generalizes the result by Cupini, Fusco and Petti in \cite{cupini.fusco.petti}. 
Eventually, we use a boot-strapping argument to get the maximal higher fractional differentiability.\\
The main result of the paper is the following.

\begin{thm}\label{mainthm}
Let $u \in W^{1,p}(\Omega)$ be the solution to the obstacle problem \eqref{obpro}, with $F$ defined by \eqref{integrand}, under Assumptions 1 and 2, for exponents $2 \leq p < \frac{n}{\alpha}$, $p<q$ verifying
$$\dfrac{q}{p}< 1 + \dfrac{\alpha}{n}.$$ If $D \psi \in B^{\gamma}_{2q-p,\infty,\text{loc}}(\Omega)$, for $0< \alpha < \gamma < 1$, then there exists $\tilde{\sigma}:=\tilde{\sigma}(p,q,n,\alpha,\beta) \in (0,1)$ s.t.
$$  V_p(Du), \ \sqrt{a(x)}V_q(Du) \in B^{t}_{2,\infty, \text{loc}}(\Omega), \quad \forall t \in ( 0, \tilde{\sigma} ).$$
\end{thm}

The paper is organized as follows. After recalling some notation and preliminary results in Section \ref{notprel}, we focus on deriving the intermediate steps that will put us in the position to prove our main result, Theorem \ref{mainthm}. In particular, in Section \ref{higher_int}, we show that the solution to the freezed obstacle problem \eqref{frozen} satisfies a variational inequality and moreover we present interior and up to the boundary higher integrability properties, which will be crucial for the comparison argument, as already mentioned. In Section \ref{hdcm}, we prove the higher fractional differentiability of the solution to the freezed obstacle problem \eqref{frozen}. We remark that the procedure used in order to do so requires the assumption $p\geq 2$. The comparison argument is presented in Section \ref{seccomp}. Finally, in Section \ref{mainsec}, we  show that a suitable fractional differentiability property on the gradient of the obstacle transfers to a higher fractional differentiability for the gradient of the minimizer, so that we are eventually able to prove Theorem \ref{mainthm}. 

We point out that, in order to prove the higher integrability of the solution to the original obstacle problem (see Theorem \ref{higherint}) and the higher differentiability of the solution to the freezed obstacle problem in Section \ref{hdcm}, Assumption 1 $(ii)$ is the only one needed on the function $b(x,w)$. On the other hand, in order to prove the comparison lemma (see Lemma \ref{comparison}), we require Assumption 2 on the coefficient $b(x,w)$.

\section{Notations and preliminary results}\label{notprel}

In what follows, $B(x,r)=B_{r}(x)= \{ y \in \mathbb{R}^{n} : |y-x | < r  \}$ will denote the ball centered at $x$ of radius $r$. We shall omit the dependence on the center and on the radius when no confusion arises. For a function $u \in L^{1}(B)$, the symbol
\begin{center}
$u_B :=\displaystyle\fint_{B} u(x) \dd x = \dfrac{1}{|B|} \displaystyle\int_{B} u(x) \dd x$.
\end{center}
will denote the integral mean of the function $u$ over the set $B$.

It is convenient to introduce an auxiliary function
\begin{center}
$V_{d}(\xi)=|\xi|^\frac{d-2}{2} \xi$
\end{center}
defined for all $\xi\in \mathbb{R}^{n}$. One can easily check that
\begin{gather}
|\xi|^d =|V_d(\xi)|^2. \label{Vp}
\end{gather}
For the auxiliary function $V_{d}$, we recall the following estimate (see the proof of \cite[Lemma 8.3]{giusti}): 
\begin{lem}\label{D1}
Let $1<d<+\infty$. There exists a constant $c=c(n,d)>0$ such that
\begin{center}
$c^{-1}(|\xi|^{2}+|\eta|^{2})^{\frac{d-2}{2}} \leq \dfrac{|V_{d}(\xi)-V_{d}(\eta)|^{2}}{|\xi-\eta|^{2}} \leq c(|\xi|^{2}+|\eta|^{2})^{\frac{d-2}{2}} $
\end{center}
for any $\xi, \eta \in \mathbb{R}^{n}, \ \xi\neq \eta$.
\end{lem}

Now we state a well-known iteration lemma (see \cite[Lemma 6.1]{giusti}  for the proof).
\begin{lem}\label{lm2}
Let $\Phi  :  [\frac{R}{2},R] \rightarrow \mathbb{R}$ be a bounded nonnegative function, where $R>0$. Assume that for all $\frac{R}{2} \leq r < s \leq R$ it holds
$$\Phi (r) \leq \theta \Phi(s) +A + \dfrac{B}{(s-r)^2}+ \dfrac{C}{(s-r)^{\gamma}}$$
where $\theta \in (0,1)$, $A$, $B$, $C \geq 0$ and $\gamma >0$ are constants. Then there exists a constant $c=c(\theta, \gamma)$ such that
$$\Phi \biggl(\dfrac{R}{2} \biggr) \leq c \biggl( A+ \dfrac{B}{R^2}+ \dfrac{C}{R^{\gamma}}  \biggr).$$
\end{lem}

\subsection{Besov-Lipschitz spaces}

Let $v:\mathbb{R}^{n} \rightarrow \mathbb{R}$ be a function. As in \cite[Section 2.5.12]{haroske}, given $0< \alpha <1$ and $1 \leq p,q < \infty$, we say that $v$ belongs to the Besov space $B^{\alpha}_{p,q}(\mathbb{R}^{n})$ if $v \in L^{p}(\mathbb{R}^{n})$ and
\begin{center}
$\Vert v \Vert_{B^{\alpha}_{p,q}(\mathbb{R}^{n})} = \Vert v \Vert_{L^{p}(\mathbb{R}^{n})} + [v]_{B^{\alpha}_{p,q}(\mathbb{R}^{n})} < \infty$,
\end{center}
where
\begin{center}
$[v]_{B^{\alpha}_{p,q}(\mathbb{R}^{n})} =  \biggl( \displaystyle\int_{\mathbb{R}^{n}} \biggl( \displaystyle\int_{\mathbb{R}^{n}} \dfrac{|v(x+h)-v(x)|^{p}}{|h|^{\alpha p}} \dd x \biggr)^{\frac{q}{p}}  \dfrac{\dd h}{|h|^{n}} \biggr)^{\frac{1}{q}}  < \infty$.
\end{center}
Equivalently, we could simply say that $v \in L^{p}(\mathbb{R}^{n})$ and $\frac{\tau_{h}{v}}{|h|^{\alpha}} \in L^{q}\bigl( \frac{\dd h}{|h|^{n}}; L^{p}(\mathbb{R}^{n}) \bigr)$. As usual, if one simply integrates for $h \in B(0, \delta)$ for a fixed $\delta >0$ then an equivalent norm is obtained, because
\begin{center}
$\biggl( \displaystyle\int_{\{|h| \geq \delta\}} \biggl( \displaystyle\int_{\mathbb{R}^{n}} \dfrac{|v(x+h)-v(x)|^{p}}{|h|^{\alpha p}} \dd x \biggr)^{\frac{q}{p}}  \dfrac{\dd h}{|h|^{n}} \biggr)^{\frac{1}{q}} \leq c(n, \alpha,p,q, \delta) \Vert v \Vert_{L^{p}(\mathbb{R}^{n})} $.
\end{center}
Similarly, we say that $v \in B^{\alpha}_{p,\infty}(\mathbb{R}^{n})$ if $v \in L^{p}(\mathbb{R}^{n})$ and
\begin{center}
$[v]_{B^{\alpha}_{p, \infty}(\mathbb{R}^{n})} =  \displaystyle\sup_{h \in \mathbb{R}^{n}} \biggl( \displaystyle\int_{\mathbb{R}^{n}} \dfrac{|v(x+h)-v(x)|^{p}}{|h|^{\alpha p}} \dd x \biggr)^{\frac{1}{p}} < \infty $.
\end{center}
Again, one can simply take supremum over $|h| \leq \delta$ and obtain an equivalent norm. By construction, $B^{\alpha}_{p, q}(\mathbb{R}^{n}) \subset L^{p}(\mathbb{R}^{n})$. One also has the following version of Sobolev embeddings (a proof can be found at \cite[Proposition 7.12]{haroske}).
\begin{lem}\label{3.1}
Suppose that $0 < \alpha <1$.
\\ (a) If $1 < p < \frac{n}{\alpha}$ and $1 \leq q \leq p^{*}_{\alpha} = \frac{np}{n- \alpha p}$, then there is a continuous embedding $B^{\alpha}_{p, q}(\mathbb{R}^{n}) \subset L^{p}(\mathbb{R}^{n})$.
\\ (b) If $p = \frac{n}{\alpha}$ and $1 \leq q \leq \infty$, then there is a continuous embedding $B^{\alpha}_{p, q}(\mathbb{R}^{n}) \subset BMO(\mathbb{R}^{n})$,
\\ where $BMO$ denotes the space of functions with bounded mean oscillations \emph{\cite[Chapter 2]{giusti}}.
\end{lem}
For further needs, we recall the following inclusions (\cite[Proposition 7.10 and Formula (7.35)]{haroske}).
\begin{lem}\label{3.2}
Suppose that $0 < \beta < \alpha < 1$.
\\ (a) If $1 < p < \infty$ and $1 \leq q \leq r \leq \infty$, then $B^{\alpha}_{p, q}(\mathbb{R}^{n}) \subset B^{\alpha}_{p, r}(\mathbb{R}^{n})$.
\\ (b) If $1 < p < \infty$ and $1 \leq q , r \leq \infty$, then $B^{\alpha}_{p, q}(\mathbb{R}^{n}) \subset B^{\beta}_{p, r}(\mathbb{R}^{n})$.
\\ (c) If $1 \leq q \leq \infty$, then $B^{\alpha}_{\frac{n}{\alpha}, q}(\mathbb{R}^{n}) \subset B^{\beta}_{\frac{n}{\beta}, q}(\mathbb{R}^{n})$.
\end{lem}
Combining Lemmas \ref{3.1} and \ref{3.2}, we get the following Sobolev type embedding theorem for Besov spaces $B^\alpha_{p,\infty}(\mathbb{R}^n)$.

\begin{lem}\label{besovembed}
Suppose that $0 < \alpha < 1$ and $1 < p < \frac{n}{\alpha}$. There is a continuous embedding $B^\alpha_{p,\infty}(\mathbb{R}^n) \subset L^{p^*_\beta}(\mathbb{R}^n)$, for every $0 < \beta < \alpha$. Moreover, the following local estimate
\begin{equation}
\Vert F \Vert _{L^{\frac{np}{n-\beta p}}(B_\varrho)} \leq c (\Vert F \Vert_{L^{p}(B_R)} + [F]_{B^{\alpha}_{p,q}(B_R)})
\end{equation}
holds for every ball $B_\varrho \subset B_R$, with $c=c(n,R,\varrho,\alpha,\beta)$.
\end{lem}

Given a domain $\Omega \subset \mathbb{R}^{n}$, we say that $v$ belongs to the local Besov space $ B^{\alpha}_{p, q,loc}$ if $\varphi \ v \in B^{\alpha}_{p, q}(\mathbb{R}^{n})$ whenever $\varphi \in \mathcal{C}^{\infty}_{c}(\Omega)$. It is worth noticing that one can prove suitable version of Lemma \ref{3.1} and Lemma \ref{3.2}, by using local Besov spaces.

The following Lemma and its proof can be found in \cite{baison.clop2017}.
\begin{lem}
A function $v \in L^{p}_{loc}(\Omega)$ belongs to the local Besov space $B^{\alpha}_{p,q,loc}$ if, and only if,
\begin{center}
$\biggl\Vert \dfrac{\tau_{h}v}{|h|^{\alpha}} \biggr\Vert_{L^{q}\bigl(\frac{\dd h}{|h|^{n}};L^{p}(B)\bigr)}<  \infty$
\end{center}
for any ball $B\subset2B\subset\Omega$ with radius $r_{B}$. Here the measure $\frac{\dd h}{|h|^n}$ is restricted to the ball $B(0,r_B)$ on the h-space.
\end{lem}

It is known that Besov-Lipschitz spaces of fractional order $\alpha \in (0,1)$ can be characterized in pointwise terms. Given a measurable function $v:\mathbb{R}^{n} \rightarrow \mathbb{R}$, a \textit{fractional $\alpha$-Hajlasz gradient for $v$} is a sequence $\{g_{k}\}_{k}$ of measurable, non-negative functions $g_{k}:\mathbb{R}^{n} \rightarrow \mathbb{R}$, together with a null set $N\subset\mathbb{R}^{n}$, such that the inequality 
\begin{center}
$|v(x)-v(y)|\leq (g_{k}(x)+g_{k}(y))|x-y|^{\alpha}$
\end{center} 
holds whenever $k \in \mathbb{Z}$ and $x,y \in \mathbb{R}^{n}\setminus N$ are such that $2^{-k} \leq|x-y|<2^{-k+1}$. We say that $\{g_{k}\}_{k} \in l^{q}(\mathbb{Z};L^{p}(\mathbb{R}^{n}))$ if
\begin{center}
$\Vert \{g_{k}\}_{k} \Vert_{l^{q}(L^{p})}=\biggl(\displaystyle\sum_{k \in \mathbb{Z}}\Vert g_{k} \Vert^{q}_{L^{p}(\mathbb{R}^{n})} \biggr)^{\frac{1}{q}}<  \infty$
\end{center} 

The following result was proved in \cite{koskela}.
\begin{thm}
Let $0< \alpha <1,$ $1 \leq p < \infty$ and $1\leq q \leq \infty $. Let $v \in L^{p}(\mathbb{R}^{n})$. One has $v \in B^{\alpha}_{p,q}(\mathbb{R}^{n})$ if, and only if, there exists a fractional $\alpha$-Hajlasz gradient $\{g_{k}\}_{k} \in l^{q}(\mathbb{Z};L^{p}(\mathbb{R}^{n}))$ for $v$. Moreover,
\begin{center}
$\Vert v \Vert_{B^{\alpha}_{p,q}(\mathbb{R}^{n})}\simeq \inf \Vert \{g_{k}\}_{k} \Vert_{l^{q}(L^{p})},$
\end{center}
where the infimum runs over all possible fractional $\alpha$-Hajlasz gradients for $v$.
\end{thm}

\subsection{Difference quotient}
We recall some properties of the finite difference quotient operator that will be needed in the sequel. Let us recall that, for every function $F:\mathbb{R}^{n}\rightarrow \mathbb{R}$ the finite difference operator is defined by
\begin{center}
$\tau_{s,h}F(x)=F(x+he_{s})-F(x)$
\end{center}
where $h \in \mathbb{R}^{n}$, $e_{s}$ is the unit vector in the $x_{s}$ direction and $s \in \{1,...,n\}$.
\\We start with the description of some elementary properties that can be found, for example, in \cite{giusti}.
\begin{prop}\label{rappincr}
Let $F$ and $G$ be two functions such that $F,G \in W^{1,p}(\Omega)$, with $p \geq1$, and let us consider the set
\begin{center}
$\Omega_{|h|} = \{ x \in \Omega : \text{dist} (x,\partial \Omega)> |h|  \}$.
\end{center}
Then
\\(i) $\tau_{h}F \in W^{1,p}(\Omega_{|h|})$ and 
\begin{center}
$D_{i}(\tau_{h}F)=\tau_{h}(D_{i}F)$.
\end{center}
(ii) If at least one of the functions $F$ or $G$ has support contained in $\Omega_{|h|}$, then
\begin{center}
$\displaystyle\int_{\Omega}F \tau_h G \dd x = \displaystyle\int_{\Omega} G \tau_{-h}F \dd x$.
\end{center}
(iii) We have $$\tau_h (FG)(x)= F(x+h)\tau_h G(x)+G(x) \tau_h F(x).$$
\end{prop}

The next result about finite difference operator is a kind of integral version of Lagrange Theorem.
\begin{lem}\label{ldiff}
If $0<\rho<R,$ $|h|<\frac{R-\rho}{2},$ $1<p<+\infty$ and $F,\ DF \in L^{p}(B_{R})$, then
\begin{center}
$\displaystyle\int_{B_{\rho}} |\tau_{h}F(x)|^{p} \dd x \leq c(n,p)|h|^{p} \displaystyle\int_{B_{R}} |DF(x)|^{p}\dd x$.
\end{center}
Moreover,
\begin{center}
$\displaystyle\int_{B_{\rho}} |F(x+h)|^{p} \dd x \leq  \displaystyle\int_{B_{R}} |F(x)|^{p}\dd x$.
\end{center}
\end{lem}

\section{Higher integrability}\label{higher_int}
The results contained in this section will be crucial for the comparison argument presented in Section \ref{seccomp}. 

Let $u$ be a solution to the obstacle problem \eqref{obpro} and fix a ball $B=B_{\frac{R}{2}}(x_0) \Subset \Omega $, for a given radius $R> 0$ and $x_0 \in \Omega$. Let us consider the so-called "freezed" functional
\begin{equation}\label{froint}
\displaystyle\int_{B} \tilde{F}(x,Dw) \dd x = \displaystyle\int_{B} b(x_0,u_B)H(x,Dw) \dd x,
\end{equation}
where $H$ was defined in \eqref{integrandh},
and let $v \in u + W^{1,p}_{0}(B)$ be the solution to
\begin{equation}\label{frozen}
\min \biggl\{ \displaystyle\int_{B} \tilde{F}(x,Dw) \dd x \ : \ w \in \mathcal{K}_{\psi}(\Omega), \ w=u \ \text{on} \ \partial B  \biggr\}.
\end{equation}
Now, we show that a local minimizer of functional \eqref{froint} satisfies a variational inequality. More precisely, we have 
\begin{prop}
A function $v \in u + W^{1,p}_{0}(B)$ is a solution to \eqref{frozen} if and only if it satisfies the following variational inequality
\begin{equation}\label{varin}
\displaystyle\int_{B} \langle D_{z}H(x,Dv)  ,D(\varphi-v) \rangle \dd x \geq 0 ,
\end{equation}
for every $\varphi \in u+W^{1,p}_0(B) \cap \mathcal{K}_{\psi}(\Omega)$ such that $H(x,D\varphi) \in L^{1}(B)$.
\end{prop}
\begin{proof}
We set $g = v+\varepsilon(\fhi -v)$ for $\varepsilon \in (0,1)$, which belongs to the obstacle class, indeed
\begin{align*}
g = v+\varepsilon(\fhi -v) = \varepsilon \fhi +(1-\varepsilon) v \geq \psi.
\end{align*}
We first notice that $H(x,D(v+\varepsilon(\fhi-v)))\in L^1$. Moreover,
\begin{align*}
\iB H(x,Dv) \dd x \leq \iB H(x,Dv+\varepsilon D(\fhi-v)) \dd x,
\end{align*}
which leads to
\begin{align*}
\iB H(x,Dv+\varepsilon D(\fhi-v)) \dd x-\iB H(x,Dv) \dd x \geq 0.
\end{align*}
From Lagrange's theorem, for $\theta \in (0,1)$ it holds
\begin{align*}
\iB \langle D_z H(x,Dv+\varepsilon \theta D(\fhi-v)), \varepsilon D(\fhi-v)\rangle \dd x\geq 0.
\end{align*}
Since $\varepsilon > 0$,
\begin{align}\label{varineqfr1}
\iB \langle D_z H(x,Dv+\varepsilon \theta D(\fhi-v)), D(\fhi-v)\rangle \dd x\geq 0.
\end{align}
According to \cite[Lemma 2.2]{coscia}, it holds
\begin{align*}
|\langle D_z H(x,z),\lambda\rangle|\leq C\left( H(x,z)+H(x,\lambda)\right).
\end{align*}
Therefore,
\begin{align}\nonumber
&|\langle D_z H(x,Dv+\varepsilon \theta D(\fhi-v)),D(\fhi-v)\rangle|\\ \nonumber
\leq &C\left( H(x,Dv+\varepsilon \theta D(\fhi-v))+H(x,D(\fhi-v))\right)\\ \nonumber
\leq &C\left( H(x,Dv)+H(x,\varepsilon \theta D(\fhi-v))+H(x,D\fhi) +H(x,Dv)\right)\\ \label{epsthet}
\leq &C\left( H(x,Dv)+(\varepsilon \theta)^p H(x,D\fhi) +(\varepsilon \theta)^p H(x,Dv)+H(x,D\fhi) \right),
\end{align}
where in the last passage we also used the direct property $H(x,\varepsilon\theta z) \leq (\varepsilon \theta)^p H(x,z)$.
For $\varepsilon \to 0$, the second and the third term on the right hand side of \eqref{epsthet} go to zero. Hence, the right hand side tends to $C\left( H(x,Dv) +H(x,D\fhi)\right)$ in $L^1$. Then, we can pass to the limit for $\varepsilon \to 0$ in \eqref{varineqfr1} applying the Dominated convergence theorem, which concludes the proof.
\end{proof}

If $u$ is a solution to \eqref{obpro}, then we are able to establish for $u$ a higher integrability result.

\begin{thm}\label{higherint}
Let $u$ be a solution to the obstacle problem \eqref{obpro} where the integrand satisfies Assumption 1, for exponents $2 \leq p < q$ verifying
$$\dfrac{q}{p}< 1 + \dfrac{\alpha}{n}.$$ 
If the function $\psi$ is s.t. $H(x,D \psi) \in L^{m_1}_{\text{loc}}(\Omega)$, for some $m_1 > 1$, then there exist an exponent $m_1 > m_2 >1$ and a positive constant $C$ s.t. it holds
\begin{align*}
\left(\fibr2 (H(x,Du))^{m_2} \dd x\right)^\frac{1}{m_2} 
&\leq C \left[ \fibR H(x,Du)\dd x + \left(\fibR (H(x,D\psi))^{m_1} \dd x \right)^\frac{1}{m_1}\right].
\end{align*}
for all balls $B_{\frac{R}{2}} \subset B_R \Subset \Omega$.
\end{thm}
\begin{proof}
Let $\dfrac{R}{2}\leq t<s\leq R\leq 1$ and let $\eta \in C_0^\infty(B_R)$ be a cut-off function s.t. $0\leq \eta\leq 1, \eta \equiv 1 \textrm{ on } B_t, \eta \equiv 0 \textrm{ outside } B_s, |D\eta| \leq \dfrac{2}{s-t}$. We set $\fhi = \eta(x)(u(x)-u_{B_R})-\eta(x)(\psi(x)-\psi_{B_R})$ and $g = u-\fhi \in \mathcal{K}_\psi(\Omega)$. We observe that $g =u$ on $\partial B_s$ and $g=\psi-\psi_{B_R}+u_{B_R}$ on $B_t$, therefore $Dg = D \psi$ on $B_t$. Using Assumption 1 (ii) and the fact that $u$ is a local minimizer, we have
\begin{align*}
&\quad \ibt H(x,Du(x))\dd x\\
&\leq C \ibt F(x,u(x),Du(x))\dd x\\
&\leq C \ibs F(x,g(x),Dg(x))\dd x\\
&\leq C \ibs |Dg(x)|^p + a(x)|Dg(x)|^q \dd x\\
&\leq C \ibs \left[ |D\eta(x)|(\psi(x)-\psi_{B_R})+ \eta(x)|D\psi(x)| + |D\eta(x)|(u(x)-u_{B_R})+(1-\eta(x))|Du(x)|\right]^p\\
&\quad + a(x)\left[ |D\eta(x)|(\psi(x)-\psi_{B_R})+ \eta(x)|D\psi(x)| + |D\eta(x)|(u(x)-u_{B_R})+(1-\eta(x))|Du(x)|\right]^q \dd x\\
&\leq C \ibs (1-\eta(x))^p\left(|Du|^p+a(x)|Du|^q\right)\dd x\\
&\quad +C\ibs \left[\left\vert \dfrac{u(x)-u_{B_R}}{s-t}\right\vert^p + a(x)\left\vert \dfrac{u(x)-u_{B_R}}{s-t}\right\vert^q\right] \dd x\\
&\quad +C\ibs \left[\left\vert \dfrac{\psi(x)-\psi_{B_R}}{s-t}\right\vert^p + a(x)\left\vert \dfrac{\psi(x)-\psi_{B_R}}{s-t}\right\vert^q\right] \dd x\\
&\quad +C \ibs \left(|D\psi(x)|^p+a(x)|D\psi(x)|^q\right) \dd x\\
&\leq C\int_{B_s\setminus B_t}H(x,Du(x))\dd x\\
&\quad + \dfrac{C}{|s-t|^p}\ibR |u(x)-u_{B_R}|^p \dd x+\dfrac{C}{|s-t|^q}\ibR a(x)|u(x)-u_{B_R}|^q \dd x\\
&\quad + \dfrac{C}{|s-t|^p}\ibR |\psi(x)-\psi_{B_R}|^p \dd x+\dfrac{C}{|s-t|^q}\ibR a(x)|\psi(x)-\psi_{B_R}|^q \dd x\\
&\quad + C\ibR H(x,D\psi(x))\dd x.
\end{align*}
Adding the quantity $C \ibt H(x,Du(x)) \dd x$ to both sides of the previous estimate, by Lemma \ref{lm2} we get
\begin{align*}
\ibr2 H(x,Du(x)) \dd x \leq & C\left[\dfrac{1}{R^p}\ibR|u(x)-u_{B_R}|^p \dd x+ \dfrac{1}{R^q}\ibR a(x)|u(x)-u_{B_R}|^q \dd x \right.\\
&+\dfrac{1}{R^p}\ibR|\psi(x)-\psi_{B_R}|^p \dd x+ \dfrac{1}{R^q}\ibR a(x)|\psi(x)-\psi_{B_R}|^q \dd x\\
&\left.+ \ibR H(x,D\psi(x))\dd x\right].
\end{align*} 
Setting $\tilde{H}(x,u(x)):= |u(x)|^p + a(x)|u(x)|^q$ and $\tilde{H}(x,\psi(x)):= |\psi(x)|^p + a(x)|\psi(x)|^q$, we can write the previous inequality as
\begin{align}\nonumber
&\fibr2 H(x,Du(x)) \dd x \\ \label{cac1int}
\leq &\fibR \tilde{H}\left(x,\dfrac{u(x)-u_{B_R}}{R}\right)\dd x+\fibR \tilde{H}\left(x,\dfrac{\psi(x)-\psi_{B_R}}{R}\right)\dd x + \fibR H(x,D\psi(x)) \dd x.
\end{align}
According to \cite[Theorem 1.6]{colombo.mingione} and H\"{o}lder's inequality, it holds
\begin{align}\nonumber
\fibR \tilde{H}\left(x,\dfrac{u(x)-u_{B_R}}{R}\right)\dd x \leq &\left(\fibR \left(\tilde{H}\left(x,\dfrac{u(x)-u_{B_R}}{R}\right)\right)^{d_1}\dd x\right)^\frac{1}{d_1}\\ \label{mingpoincare1}
\leq &\left(\fibR \left(H(x,Du(x))\right)^{d_2}\dd x\right)^\frac{1}{d_2},
\end{align}
where $d_2<1<d_1$ depend on $n,p,q,\alpha$. Analogously, 
\begin{align}\nonumber
\fibR \tilde{H}\left(x,\dfrac{\psi(x)-\psi_{B_R}}{R}\right)\dd x \leq &\left(\fibR \left(\tilde{H}\left(x,\dfrac{\psi(x)-\psi_{B_R}}{R}\right)\right)^{d_1}\dd x\right)^\frac{1}{d_1}\\ \label{mingpoincare2}
\leq &\left(\fibR \left(H(x,D\psi(x))\right)^{d_2}\dd x\right)^\frac{1}{d_2}.
\end{align}
Inserting \eqref{mingpoincare1} and \eqref{mingpoincare2} in \eqref{cac1int} and exploiting H\"{o}lder's inequality, we infer
\begin{align}\label{cacd2}
\fibr2 H(x,Du(x)) \dd x \leq C \left[ \left(\fibR (H(x,Du(x)))^{d_2}\right)^\frac{1}{d_2} + \fibR H(x,D\psi(x)) \dd x\right].
\end{align}
Since $H(x,D\psi(x)) \in L^{m_1}$, for $m_1>1$, from Gehring's lemma proved in \cite{giusti} it follows that there exists $m_1>m_2>1$ s.t. $H(x,Du(x))\in L^{m_2}$. Then, holding to $d_2 <1$, we might write
\begin{align*}
\fibr2 (H(x,Du(x)))^{m_2} \dd x 
&\leq C \left[ \left(\fibR H(x,Du(x)) \dd x \right)^{m_2} + \fibR (H(x,D\psi(x)))^{m_2}\dd x \right]\\
&\leq C \left[ \left(\fibR H(x,Du(x)) \dd x \right)^{m_2} + \left(\fibR (H(x,D\psi(x))^{m_1} \dd x \right)^\frac{m_2}{m_1} \right].
\end{align*}
Hence,
\begin{align*}
\left(\fibr2 (H(x,Du(x)))^{m_2} \dd x\right)^\frac{1}{m_2} 
&\leq C \left[ \fibR H(x,Du(x))\dd x + \left(\fibR (H(x,D\psi(x))^{m_1} \dd x \right)^\frac{1}{m_1}\right].
\end{align*}
\end{proof}

The higher integrability of the minimizer $u$ stated in Theorem \ref{higherint} allows us to prove the following higher integrability up to the boundary result for the solution to the freezed obstacle problem \eqref{frozen}.
\begin{thm}\label{higint2}
Let $v \in u+W^{1,p}_0(B_{\frac{R}{2}})$ be a solution to the obstacle problem \eqref{frozen} where the integrand $\tilde{F}$ satisfies Assumption 1, for exponents $2 \leq p < q$ verifying
$$\dfrac{q}{p}< 1 + \dfrac{\alpha}{n}.$$
If the function $\psi$ is s.t. $H(x,D\psi) \in L^{m_1}_{\text{loc}}(\Omega)$, for some $m_1>1$, then $H(x,Du)  \in L^{m_2}_{\text{loc}}(\Omega)$, for some $m_1>m_2>1$, and there exist a constant $C$ and an exponent $m_3$, with $m_1>m_2>m_3>1$, s.t. $H(x,Dv)  \in L^{m_3}_{\text{loc}}(\Omega)$ and
\begin{align*}
\left(\fint_{B_\frac{R}{2}} \left(H(x,Dv)\right)^{m_3}\dd x \right)^\frac{1}{m_3}
\leq C\left[\left(\fibR (H(x,Du))^{m_2}\dd x \right)^\frac{1}{m_2}
+ \left(\fibR (H(x,D\psi))^{m_2} \dd x \right)^\frac{1}{m_2}\right].
\end{align*}

\end{thm}

\begin{proof}
We start setting 
\begin{equation}\label{defw}
w(x) := 
\left\lbrace{
\begin{array}{l}
v(x) \qquad \textrm{if } x\in B_\frac{R}{2},  
\\
u(x) \qquad \textrm{if } x\in B_R \setminus B_{\frac{R}{2}}
\end{array}}
\right.
\end{equation}
We first consider $B_\rho(x_1) \subset B_\frac{R}{2}$. In this case the Caccioppoli inequality \eqref{cacd2} holds, namely
\begin{align}\label{cadis}
\fibro2 H(x,Dv)\dd x \leq C \left[ \left(\fibrho (H(x,Dv))^{d_2} \dd x \right)^\frac{1}{d_2} + \fibrho H(x,D\psi) \dd x\right].
\end{align}
Let us now focus on the case $B_\rho(x_1) \subset B_R$, with $x_1 \in \partial B_\frac{R}{2}$. Fix $\frac{\rho}{2}\leq t < s \leq \rho$ and a cut-off function $\eta$ between $B_s(x_1)$ and $B_t(x_1)$, with $|D\eta| \leq \dfrac{2}{t-s}$.
Let us set $g(x) := (1-\eta(x)) v +\eta(x) u(x)$. It is straightforward that $g \in u+W_0^{1,p}$ and $g(x)\geq \psi(x)$. Since $v$ is a minimizer, according to the definition of $H$ and Assumption 1 (ii), we have
\begin{align*}
\int_{B_t \cap B_{\frac{R}{2}}} H(x,Dv)\dd x
&\leq C \int_{B_t \cap B_{\frac{R}{2}}}\tilde{F}(x,Dv) \dd x\\
&\leq C \int_{B_s \cap  B_{\frac{R}{2}}}\tilde{F}(x,Dg)\dd x.
\end{align*}
Therefore, from the definitions of $g$ and $\eta$, we get
\begin{align*}
&\int_{B_t\cap B_{\frac{R}{2}}} H(x,Dv)\dd x \\
\leq &C\left[ \int_{B_s\cap B_{\frac{R}{2}}}\left(\dfrac{1}{(t-s)^p}|u-v|^p + a(x)\dfrac{1}{(t-s)^q}|u-v|^q\right)\dd x\right.\\
& + \int_{(B_s \setminus B_t)\cap B_{\frac{R}{2}}}H(x,Dv)\dd x\\
& \left. + \ibs H(x,Du)\dd x\right].
\end{align*}
As before, adding the quantity $C\int_{B_t \cap B_{\frac{R}{2}}} H(x,Dv) \dd x$ to both sides of the previous inequality, by Lemma \ref{lm2} we get
\begin{align}\label{riass2}
&\fint_{B_\frac{\rho}{2}\cap B_{\frac{R}{2}}} H(x,Dv)\dd x \notag \\
\leq &C\left[ \fint_{B_\rho\cap B_{\frac{R}{2}}}\left(\dfrac{1}{\rho^p}|u-v|^p + a(x)\dfrac{1}{\rho^q}|u-v|^q\right)\dd x\right. \notag\\
& \left. + \fibrho H(x,Du)\dd x\right].
\end{align}
We set 
\begin{align*}
 \tilde{H}\left(x,\dfrac{u-v}{\rho}\right)
:= \dfrac{1}{\rho^p}|u-v|^p + a(x)\dfrac{1}{\rho^q}|u-v|^q.
\end{align*}
Exploiting the Poincar\'{e} inequality in \cite[Remark 2]{colombo.mingione} and H\"{o}lder's inequality, we have
\begin{align}\nonumber
\fint_{B_\rho \cap B_\frac{R}{2}} \tilde{H}\left(x,\dfrac{u-v}{\rho}\right)\dd x 
\leq &\left(\fint_{B_\rho\cap B_{\frac{R}{2}}} \left(\tilde{H}\left(x,\dfrac{u-v}{\rho}\right)\right)^{d_1}\dd x\right)^\frac{1}{d_1}\\ \label{mingpoincare3}
\leq &\left(\fint_{B_\rho\cap B_{\frac{R}{2}}} \left(H(x,Du-Dv)\right)^{d_2}\dd x\right)^\frac{1}{d_2},
\end{align}
where $d_2<1<d_1$ depend on $n,p,q,\alpha$.
Inserting \eqref{mingpoincare3} in \eqref{riass2}, it yields
\begin{align*}
\fint_{B_\frac{\rho}{2}\cap B_{\frac{R}{2}}} H(x,Dv)\dd x 
\leq &C\left[ \left(\fint_{B_\rho}\left(H(x,Du(x)\right)^{d_2}\dd x\right)^\frac{1}{d_2}\right.\\
&+ \left(\fint_{B_\rho \cap B_\frac{R}{2}}\left(H(x,Dv)\right)^{d_2}\dd x\right)^\frac{1}{d_2}\\
& \left. + \fibrho H(x,Du(x))\dd x\right]\\
\leq &C\left[\left(\fint_{B_\rho \cap B_\frac{R}{2}}\left(H(x,Dv)\right)^{d_2}\dd x\right)^\frac{1}{d_2}
+ \fibrho H(x,Du(x))\dd x\right].
\end{align*}
Therefore, from the definition of $w$ in \eqref{defw}, we infer
\begin{align}\label{estpal}
\fint_{B_\frac{\rho}{2}} H(x,Dw(x))\dd x 
\leq &C\biggl[\left(\fint_{B_\rho}\left(H(x,Dw(x))\right)^{d_2}\dd x\right)^\frac{1}{d_2}
+ \fibrho H(x,Du(x))\dd x \notag\\
&+ \fibrho H(x,D\psi(x))\dd x \biggr].
\end{align}
Hence, by \eqref{cadis} it follows that \eqref{estpal} holds not only if $B_\rho(x_1) \subset B_{\frac{R}{2}}$ or $B_\rho(x_1) \cap B_{\frac{R}{2}} \neq \emptyset$, but also when $B_\rho(x_1) \subset B_R$ and $x_1 \in \partial B_\frac{R}{2}$.

We now take care of the case $B_\rho(x_1) \cap \partial B_\frac{R}{2} \neq \emptyset$ and $B_{4\rho}\subset B_R$.
We fix $x_2 \in B_\rho(x_1) \cap \partial B_\frac{R}{2}$.
\begin{align*}
&\fint_{B_\frac{\rho}{2}(x_1)}H(x,Dw)\dd x\\
 &\leq 3^N \fint_{B_\frac{3\rho}{2}(x_2)}H(x,Dw) \dd x\\
&\leq C\left[ \left(\fint_{B_{3\rho(x_2)}}(H(x,Dw))^{d_2}\dd x \right)^\frac{1}{d_2} + \fint_{B_{3\rho}(x_2)}H(x,Du)\dd x 
+ \fint_{B_{3\rho}(x_2)}H(x,D\psi)\dd x \right]\\
&\leq C \left[\left(\fint_{B_{4\rho(x_1)}}(H(x,Dw))^{d_2}\dd x \right)^\frac{1}{d_2} + \fint_{B_{4\rho}(x_1)}H(x,Du)\dd x 
+ \fint_{B_{4\rho}(x_1)}H(x,D\psi)\dd x \right].
\end{align*}
Since this estimate holds for every $B_\frac{\rho}{2}$ such that $B_{4\rho}\subset B_R$, by a covering argument it follows that inequality \eqref{estpal} holds for every $B_\frac{\rho}{2}$ such that $B_{\rho}\subset B_R$.
Now, since $H(x,D\psi) \in L^{m_1},\ m_1>1$, Theorem \ref{higherint} yields that there exists $m_2$, with $1<m_2<m_1$, s.t. $H(x,Du) \in L^{m_2}$. Therefore, according to Gehring's lemma, there exists $m_3$, with $1<m_3<m_2<m_1$, such that, 
\begin{align*}
&\left(\fint_{B_\frac{\rho}{2}(x_1)} \left(H(x,Dw)\right)^{m_3}\dd x  \right)^\frac{1}{m_3}\\
\leq & C\left[\fint_{B_\rho (x_1)}H(x,Dw(x))\dd x \right.\\
&\left.+ \left(\fint_{B_\rho (x_1)} (H(x,Du(x)))^{m_2}\dd x \right)^\frac{1}{m_2}
+ \left(\int_{B_\rho (x_1)} (H(x,D\psi(x)))^{m_2} \dd x \right)^\frac{1}{m_2}\right].
\end{align*}
In particular, for $\rho \equiv R$ and $x_1 = x_0$, recalling the definition of $w$ we have

\begin{align*}
&\left(\fint_{B_\frac{R}{2}} \left(H(x,Dv)\right)^{m_3}\dd x \right)^\frac{1}{m_3}\\
\leq &C\left[\fint_{B_\frac{R}{2}}H(x,Dv)\dd x 
+ \fint_{B_R\setminus B_\frac{R}{2}}H(x,Du(x))\dd x \right.\\
&+ \left.\left(\fibR (H(x,Du(x)))^{m_2} \dd x \right)^\frac{1}{m_2}
+ \left(\fibR (H(x,D\psi(x)))^{m_2} \dd x \right)^\frac{1}{m_2}\right]\\
\leq &C\left[\fint_{B_\frac{R}{2}}H(x,Dv)\dd x \right.\\
&+ \left.\left(\fibR (H(x,Du(x)))^{m_2} \dd x \right)^\frac{1}{m_2}
+ \left(\fibR (H(x,D\psi(x)))^{m_2} \dd x \right)^\frac{1}{m_2}\right].\\
\end{align*}
Since $v$ is a minimizer and recalling that $m_2>1$, it holds

\begin{align*}
&\left(\fint_{B_\frac{R}{2}} \left(H(x,Dv)\right)^{m_3}\dd x \right)^\frac{1}{m_3}\\
&\leq C\left[\left(\fibR (H(x,Du(x)))^{m_2}\dd x \right)^\frac{1}{m_2}
+ \left(\fibR (H(x,D\psi(x)))^{m_2} \dd x \right)^\frac{1}{m_2}\right],
\end{align*}
i.e. the conclusion.
\end{proof}

\begin{rmk}
We point out that Theorems \ref{higherint} and \ref{higint2} hold true also under the more general hypothesis $q>p>1$. However, they are stated for $q>p\geq 2$ for later purpose in Section \ref{mainsec}.
\end{rmk}

\section{Higher differentiability for comparison maps}\label{hdcm}
The higher differentiability of the solution $v$ to \eqref{frozen} has been already established in \cite{grimaldi.ipocoana} under more general assumptions on the coefficients. The strategy relied on the combination of approximation results and a priori estimates. Here, we only give the proof of the a priori bounds, in order to establish precise estimates on the difference quotient that will be crucial for the comparison argument. On the other hand, the approximation procedure is achieved using the same arguments in \cite{grimaldi.ipocoana}, therefore it will not be presented.  \\
Before stating the result, it is worth noticing that Assumption 1 implies that there exist positive constants $\tilde{l},\tilde{\nu}, \tilde{L}$ such that the following conditions are satisfied:
$$ |D_\xi\tilde{F}(x, \xi)| \leq \tilde{l} (|\xi|^{p-1}+a(x)|\xi|^{q-1}) \eqno{\rm{{ (A1)}}}  $$
$$  \langle D_\xi\tilde{F}(x,\xi)- D_\xi\tilde{F}(x,\eta), \xi-\eta  \rangle \geq \tilde{\nu} (|\xi-\eta|^{2} (|\xi|^{2}+|\eta|^{2})^{\frac{p-2}{2}} +a(x)|\xi-\eta|^{2} (|\xi|^{2}+|\eta|^{2})^{\frac{q-2}{2}}) \eqno{\rm{{ (A2)}}}$$
$$ |D_\xi\tilde{F}(x,\xi)-D_\xi\tilde{F}(x,\eta)|\leq \tilde{L}(|\xi-\eta| (|\xi|^{2}+|\eta|^{2})^{\frac{p-2}{2}}+a(x)|\xi-\eta| (|\xi|^{2}+|\eta|^{2})^{\frac{q-2}{2}})\eqno{\rm{{ (A3)}}}$$
$$ |D_{\xi}\tilde{F}(x, \xi)- D_{\xi}\tilde{F}(y, \xi)| \leq |x-y|^\alpha|\xi|^{q-1}\eqno{\rm{{ (A4)}}}$$
for every $x ,y\in \Omega$ and every $\xi, \eta \in \mathbb{R}^n$.

The following lemma holds:
\begin{lem}\label{highcomp}
Let $v \in u+ W^{1,p}_0(B)$ be the solution to \eqref{frozen} under Assumption 1, for exponents $2 \leq p <  \frac{n}{\alpha}$, $p<q$ satisfying
\begin{equation}\label{gap}
\dfrac{q}{p} < 1 + \dfrac{\alpha}{n}.
\end{equation}
If $$D \psi \in B^{\gamma}_{2q-p,\infty}(B),$$ for $0 < \alpha < \gamma < 1$, then $$V_p(Dv) \in B^{\alpha}_{2,\infty, \text{loc}}(B)$$ and the following estimate
\begin{align}\label{diffquotient}
&  \displaystyle\int_{B_{r/4}}(|\tau_{h}V_p(Dv)|^{2}+a(x+h)|\tau_{h}V_q(Dv)|^{2})\dd x \notag\\
 \leq & C |h|^{2\alpha} \biggl\{ \dfrac{1}{r^{2 \tilde{p}}} \biggl(\displaystyle\int_{B_r}  (1+|Dv|^p+|D\psi|^{2q-p})\dd x \biggr)^\kappa+ [ D \psi ]^{2q-p}_{B^{\gamma}_{2q-p, \infty}(B_r)} \biggl\},
\end{align}
holds for all balls $B_{r/4} \subset B_r \Subset B$, for some $\mu < \alpha$, with $C := C(n,p,q, \mu, \Vert a \Vert_{\infty},\Vert D\psi \Vert_{B^{\gamma}_{2q-p, \infty}})$, $\tilde{p}:=\tilde{p}(n,p,q,\mu)>1$ and  $\kappa:= \kappa (n,p,q, \mu)<\tilde{p}$.
\end{lem}
\proof We a priori assume that $Dv \in L^{\frac{np}{n-2 \mu}}_{\text{loc}}(B)$, for all $ \frac{\alpha n}{n+2\alpha }< \mu < \alpha$. 
\\In the sequel we will profusely use the following inequality:
\begin{equation}\label{dis}
2q-p \leq \dfrac{np}{n-2 \mu} ,
\end{equation}
for $\mu \in [\frac{\alpha n}{n+2\alpha }, \alpha)$. Indeed,
$$2q-p \leq \dfrac{np}{n-2 \mu} \Leftrightarrow \dfrac{q}{p} \leq \dfrac{n-\mu}{n-2 \mu}$$
and $$1+ \dfrac{\alpha}{n} \leq \dfrac{n-\mu}{n-2 \mu} \Leftrightarrow \mu \geq \dfrac{\alpha n}{n+2\alpha }.$$
Fix $0 < \frac{r}{4} < \rho < s < t < t' < \frac{r}{2} $ such that $B_{r} \Subset B$ and a cut-off function $\eta \in \mathcal{C}_0^1(B_t)$ such that $0 \leq \eta \leq 1$, $\eta =1$ on $B_s$, $|D \eta | \leq \frac{C}{t-s}$.
\\Now, for $|h| < \frac{r}{4}$, we consider functions
\begin{gather}
w_{1}(x)= \eta^{2}(x) [(v-\psi)(x+h)-(v-\psi)(x)] \notag 
\end{gather}
and
\begin{gather}
w_{2}(x)= \eta^{2}(x-h) [(v-\psi)(x-h)-(v-\psi)(x)]. \notag 
\end{gather}
Then
\begin{equation}
\varphi_1(x)=v(x)+tw_1(x), \label{2:2} 
\end{equation}
\begin{equation}
\varphi_2(x)=v(x)+tw_2(x) \label{2:3}
\end{equation}
are admissible test functions for all $t \in [0,1)$.
\\Arguing analogously as in the proof of \cite[Theorem 4.1]{grimaldi.ipocoana}, we obtain the following estimate
\begin{align}
0 \geq & \displaystyle\int_{\Omega} \langle D_\xi H(x+h,Dv(x+h))-D_\xi H(x+h,Dv(x)),\eta^{2}D\tau_{h}v \rangle \dd x\notag\\
  &-\displaystyle\int_{\Omega} \langle D_\xi H(x+h,Dv(x+h))-D_\xi H(x+h,Dv(x)),\eta^{2}D\tau_{h}\psi \rangle \dd x\notag\\
  &+\displaystyle\int_{\Omega} \langle D_\xi H(x+h,Dv(x+h))-D_\xi H(x+h,Dv(x)),2\eta D \eta\tau_{h}(v-\psi) \rangle \dd x\notag\\
  &+\displaystyle\int_{\Omega} \langle D_\xi H(x+h,Dv(x))-D_\xi H(x,Dv(x)),\eta^{2}D\tau_{h}v \rangle \dd x\notag\\
  &-\displaystyle\int_{\Omega} \langle D_\xi H(x+h,Dv(x))-D_\xi H(x,Dv(x)),\eta^{2}D\tau_{h}\psi \rangle \dd x\notag\\
  &+\displaystyle\int_{\Omega} \langle D_\xi H(x+h,Dv(x))-D_\xi H(x,Dv(x)),2\eta D \eta\tau_{h}(v-\psi) \rangle \dd x \notag\\
 =:& I_{1}+I_{2}+I_{3}+I_{4}+I_{5}+I_{6}, \label{2:7}
\end{align}
that yields
\begin{align}
I_1 \leq & |I_2| + |I_3| + |I_4| + |I_5| + |I_6|. \label{2:8}
\end{align}
%Now, we take care of the integrals $I_1$, $I_2$, $I_3$. The calculations performed in [Theorem 4.1, GI] lead us to the following estimates 
The ellipticity assumption (A2) and the properties of $a(x)$ imply 
\begin{align}
I_{1} \geq & \tilde{\nu} \displaystyle\int_{\Omega}  \eta^{2} |\tau_{h}Dv|^{2}( |Dv(x+h)|^{2}+|Dv(x)|^{2})^{\frac{p-2}{2}} \dd x \notag\\
& + \tilde{\nu} \displaystyle\int_{\Omega}  \eta^{2} a(x+h) |\tau_{h}Dv|^{2}( |Dv(x+h)|^{2}+|Dv(x)|^{2})^{\frac{q-2}{2}} \dd x \notag\\
\geq & \tilde{\nu} \displaystyle\int_{\Omega}  \eta^{2} (|\tau_{h}V_p(Dv)|^{2}+a(x+h)|\tau_{h}V_q(Dv)|^{2})\dd x.
\label{I1}
\end{align}
From the growth condition (A3), the boundedness of $a(x)$ and Young's inequality, we get
\begin{align}
|I_{2}| \leq & \tilde{L} \displaystyle\int_{\Omega} \eta^{2} |\tau_{h}Dv|( |Dv(x+h)|^{2}+|Dv(x)|^{2})^{\frac{p-2}{2}}|\tau_{h}D \psi| \dd x \notag \\
& + \tilde{L} \displaystyle\int_{\Omega} \eta^{2} a(x+h) |\tau_{h}Dv|( |Dv(x+h)|^{2}+|Dv(x)|^{2})^{\frac{q-2}{2}}|\tau_{h}D \psi| \dd x \notag \\
\leq & \tilde{L} \displaystyle\int_{\Omega} \eta^{2} |\tau_{h}Dv|( |Dv(x+h)|^{2}+|Dv(x)|^{2})^{\frac{p-2}{2}}|\tau_{h}D \psi| \dd x \notag \\
& + \tilde{L} \Vert a \Vert_{\infty} \displaystyle\int_{\Omega} \eta^{2}|\tau_{h}Dv|( |Dv(x+h)|^{2}+|Dv(x)|^{2})^{\frac{q-2}{2}}|\tau_{h}D \psi| \dd x\notag \\
 \leq & \varepsilon \displaystyle\int_{\Omega} \eta^2 |\tau_hDv|^2 (|Dv(x+h)|^2+|Dv(x)|^2)^{\frac{p-2}{2}}\dd x \notag\\
&+ C_{\varepsilon}(\tilde{L}, \Vert a \Vert_{\infty}) \displaystyle\int_{\Omega} \eta^2 |\tau_h D \psi |^2 (1+|Dv(x+h)|^2+|Dv(x)|^2)^{\frac{2q-p-2}{2}}\dd x. \notag
\end{align}
The calculations performed in \cite[Theorem 4.1]{grimaldi.ipocoana} and Lemma \ref{D1} lead us to the following estimate for the integral $I_{2}$
\begin{align}
|I_{2}|\leq &  \varepsilon \displaystyle\int_{\Omega} \eta^2 |\tau_hV_p(Dv)|^2 \dd x \notag\\
&+ C_{\varepsilon}(\tilde{L},p,q, \Vert a \Vert_{\infty})|h|^{2\gamma}[D \psi]^{2q-p}_{B^{\gamma}_{2q-p, \infty}(B_r)} \notag\\
&+  C_{\varepsilon}(\tilde{L},p,q, \Vert a \Vert_{\infty})|h|^{2\gamma} \displaystyle\int_{B_{t'}}(1+|Dv|)^{2q-p}\dd x.\label{I2}
\end{align}
Now, we consider the integral $I_3$. From assumption (A3), hypothesis $|D\eta| \leq \frac{C}{t-s}$ and Young's inequality, we get
\begin{align}
|I_{3}|\
\leq &  2\tilde{L} \displaystyle\int_{\Omega} |D \eta| \eta |\tau_{h} Dv| ( |Dv(x+h)|^{2}+|Dv(x)|^{2})^{\frac{p-2}{2}}|\tau_{h}(v-\psi)| \dd x \notag\\
&+  2\tilde{L} \Vert a \Vert_{\infty} \displaystyle\int_{\Omega} |D \eta| \eta|\tau_{h} Dv| (1+ |Dv(x+h)|^{2}+|Dv(x)|^{2})^{\frac{q-2}{2}}|\tau_{h}(v-\psi)| \dd x \notag\\
\leq &  \varepsilon \displaystyle\int_{\Omega} \eta^2 |\tau_hDv|^2 (|Dv(x+h)|^2+|Dv(x)|^2)^{\frac{p-2}{2}}\dd x \notag\\
&+ \dfrac{C_{\varepsilon}(L, \Vert a \Vert_{\infty})}{(t-s)^2} \displaystyle\int_{B_t} |\tau_h(v- \psi)|^2(|Dv(x+h)|^2+|Dv(x)|^2)^{\frac{2q-p-2}{2}}\dd x, \notag
\end{align}
where we also used the boundedness of function $a(x)$.\\
Arguing analogously as in the proof of \cite[Theorem 4.1]{grimaldi.ipocoana}, we can estimate the integral $I_{3}$ as follows
\begin{align}
 |I_{3}| \leq &\varepsilon \displaystyle\int_{\Omega} \eta^2 |\tau_hV_p(Dv)|^2\dd x \notag\\
&+\dfrac{C_{\varepsilon}(\tilde{L},n,p,q, \Vert a \Vert_{\infty})}{(t-s)^2}|h|^2 \displaystyle\int_{B_r}|D \psi|^{2q-p}\dd x  \notag\\
&+\dfrac{C_{\varepsilon}(\tilde{L},n,p,q, \Vert a \Vert_{\infty})}{(t-s)^2} |h|^2 \displaystyle\int_{B_{t'}}(1+|Dv|)^{2q-p}\dd x . \label{I3}
\end{align}
In order to estimate the integral $I_4$, we use assumption (A4), Young's inequality and Lemma \ref{D1} as follows
\begin{align}
|I_4| \leq & \displaystyle\int_{\Omega} \eta^2 |\tau_h Dv| |h|^{\alpha} |Dv|^{q-1}\dd x \notag\\
\leq & \varepsilon \displaystyle\int_{\Omega} \eta^2 |\tau_hDv|^2 (|Dv(x+h)|^2+|Dv(x)|^2)^{\frac{p-2}{2}}\dd x \notag\\
&+ C_{\varepsilon}|h|^{2 \alpha} \displaystyle\int_{B_t} |Dv|^{2q-p}\dd x\notag\\
\leq & \varepsilon \displaystyle\int_{\Omega} \eta^2 |\tau_hV_p(Dv)|^2 \dd x \notag\\
& + C_{\varepsilon}|h|^{2 \alpha} \displaystyle\int_{B_t} |Dv|^{2q-p}\dd x . \label{I4}
\end{align}
We now take care of $I_5$. Similarly as above, exploiting assumption (A4) and H\"{o}lder's inequality, we infer
\begin{align*}
|I_5| 
\leq & \io \eta^2 |\tau_h D\psi| |h|^\alpha |Dv|^{q-1} \dd x\\
\leq & |h|^\alpha \left( \ibt |\tau_h D\psi|^{2q-p} \dd x \right)
^\frac{1}{2q-p}\left( \ibt |Dv|^\frac{(q-1)(2q-p)}{2q-p-1} \dd x \right)^\frac{2q-p-1}{2q-p}.
\end{align*}
Now, we observe
\begin{align}\label{disI5}
\frac{(q-1)(2q-p)}{2q-p-1} < 2q-p \Leftrightarrow p < q .
\end{align}
Hence
\begin{align}\label{I5}
|I_5| \leq & |h|^{\alpha+\gamma}  [D \psi]_{B^{\gamma}_{2q-p, \infty}(B_r)}\left( \ibt|Dv|^{2q-p} \dd x \right)^\frac{q-1}{2q-p} \notag\\
\leq &  C(q)|h|^{\alpha+\gamma}  [D \psi]^q_{B^{\gamma}_{2q-p, \infty}(B_r)} \notag\\
& + C(q) |h|^{\alpha+\gamma} \left( \ibt|Dv|^{2q-p} \dd x \right)^\frac{q}{2q-p}.
\end{align}
From assumption (A4), hypothesis $|D\eta| \leq \frac{C}{t-s}$ and H\"{o}lder's inequality, we infer the following estimate for $I_6$.

\begin{align*}
|I_6| \leq &\frac{C}{t-s} |h|^\alpha \ibt |\tau_h \psi| |Dv|^{q-1} \dd x\\
&+\frac{C}{t-s} |h|^\alpha \ibt |\tau_h v| |Dv|^{q-1} \dd x\\
\leq &\frac{C}{t-s} |h|^\alpha  \left(\ibt |\tau_h \psi|^{2q-p} \dd x\right)^\frac{1}{2q-p} \left( \ibt |Dv|^\frac{(q-1)(2q-p)}{2q-p-1}\dd x \right)^\frac{2q-p-1}{2q-p}\\
&+\frac{C}{t-s} |h|^\alpha  \left(\ibt |\tau_h v|^{2q-p} \dd x\right)^\frac{1}{2q-p} \left( \ibt |Dv|^\frac{(q-1)(2q-p)}{2q-p-1}\dd x \right)^\frac{2q-p-1}{2q-p}.
\end{align*}
Using Lemma \ref{ldiff}, \eqref{disI5} and H\"{o}lder's  and Young's inequality, we have
\begin{align}
|I_6| \leq &\frac{C(n,p,q)}{t-s} |h|^{\alpha +1} \left(\ibtt |D\psi|^{2q-p} \dd x\right)^\frac{1}{2q-p} \left( \ibt |Dv|^\frac{(q-1)(2q-p)}{2q-p-1}\dd x \right)^\frac{2q-p-1}{2q-p} \notag\\
&+\frac{C(n,p,q)}{t-s} |h|^{\alpha +1} \left(\ibtt |Dv|^{2q-p} \dd x\right)^\frac{1}{2q-p} \left( \ibt |Dv|^\frac{(q-1)(2q-p)}{2q-p-1}\dd x \right)^\frac{2q-p-1}{2q-p} \notag\\
\leq &\frac{C(n,p,q)}{t-s} |h|^{\alpha +1} \left(\displaystyle\int_{B_r} |D\psi|^{2q-p} \dd x\right)^\frac{1}{2q-p} \left( \ibt |Dv|^{2q-p}\dd x \right)^\frac{q-1}{2q-p} \notag\\
&+\frac{C(n,p,q)}{t-s} |h|^{\alpha +1} \left(\ibtt |Dv|^{2q-p} \dd x\right)^\frac{q}{2q-p} \notag\\
\leq & \frac{C(n,p,q)}{t-s} |h|^{\alpha +1} \left(\displaystyle\int_{B_r} |D\psi|^{2q-p} \dd x\right)^\frac{q}{2q-p}  \notag\\
& +\frac{C(n,p,q)}{t-s} |h|^{\alpha +1} \left(\ibtt |Dv|^{2q-p} \dd x\right)^\frac{q}{2q-p} . \label{I6}
\end{align}
Inserting estimates \eqref{I1}, \eqref{I2}, \eqref{I3}, \eqref{I4}, \eqref{I5} and \eqref{I6} in \eqref{2:8}, we infer
\begin{align}\label{estimate1.}
& \nu  \displaystyle\int_{\Omega}  \eta^{2} (|\tau_{h}V_p(Dv)|^{2} 
+  a(x+h)|\tau_{h}V_p(Dv)|^{2}) \dd x\notag\\
\leq & 3\varepsilon \displaystyle\int_{\Omega} \eta^2 |\tau_hV_p(Dv)|^2 \dd x \notag\\
&+ C_{\varepsilon}(\tilde{L},p,q, \Vert a \Vert_{\infty})|h|^{2\gamma}[D \psi]^{2q-p}_{B^{\gamma}_{2q-p, \infty}(B_r)} \notag\\
&+  C_{\varepsilon}(\tilde{L},p,q, \Vert a \Vert_{\infty})|h|^{2\gamma} \displaystyle\int_{B_{t'}}(1+|Dv|)^{2q-p}\dd x \notag \\
&+\dfrac{C_{\varepsilon}(\tilde{L},n,p,q, \Vert a \Vert_{\infty})}{(t-s)^2}|h|^2 \displaystyle\int_{B_r}|D \psi|^{2q-p}\dd x  \notag\\
&+\dfrac{C_{\varepsilon}(\tilde{L},n,p,q, \Vert a \Vert_{\infty})}{(t-s)^2} |h|^2 \displaystyle\int_{B_{t'}}(1+|Dv|)^{2q-p}\dd x \notag\\
&+ C_{\varepsilon}|h|^{2 \alpha} \displaystyle\int_{B_t} |Dv|^{2q-p}\dd x \notag\\ 
&+C(q)|h|^{\alpha+\gamma}  [D \psi]^q_{B^{\gamma}_{2q-p, \infty}(B_r)}+C(q)|h|^{\alpha+\gamma}\left( \ibt|Dv|^{2q-p} \dd x \right)^\frac{q}{2q-p}\notag\\
&+\frac{C(n,p,q)}{t-s} |h|^{\alpha +1} \left(\displaystyle\int_{B_r} |D\psi|^{2q-p} \dd x\right)^\frac{q}{2q-p} \notag\\
&+\frac{C(n,p,q)}{t-s} |h|^{\alpha +1} \left(\ibtt |Dv|^{2q-p} \dd x\right)^\frac{q}{2q-p}.
\end{align}
We now introduce the following interpolation inequality
\begin{align}\label{interp1}
\Vert Dw \Vert_{2q-p} \leq \Vert Dw \Vert_p^{\delta} \Vert Dw \Vert_{\frac{np}{n-2 \mu}}^{1- \delta},
\end{align}
where $0< \delta <1$ is defined through the condition
\begin{align}\label{defdelta1}
\dfrac{1}{(2q-p)}= \dfrac{\delta}{p}+ \dfrac{(1-\delta)(n-2 \mu)}{np}
\end{align}
which implies
$$\delta= \dfrac{n(p-q)+\mu (2q-p)}{\mu (2q-p)}, \ \ \ 1- \delta= \dfrac{n(q-p)}{\mu (2q-p)}.$$
Hence we get the following inequalities
\begin{align}\label{interpol}
\displaystyle\int_{B_{t'}}(1+|Dv|)^{2q-p}\dd x \leq & \biggl( \displaystyle\int_{B_{t'}} (1+|Dv|)^{p}\dd x \biggr)^{\frac{\delta(2q-p)}{p}}   \biggl( \displaystyle\int_{B_{t'}}(1+|Dv|)^{\frac{np}{n-2 \mu}}\dd x \biggr)^{\frac{(n-2 \mu)(q-p)}{\mu p}},\\
\biggr(\displaystyle\int_{B_{t'}}|Dv|^{2q-p}\dd x \biggl)^{\frac{q}{2q-p}} \leq & \biggl( \displaystyle\int_{B_{t'}} |Dv|^{p}\dd x \biggr)^{\frac{\delta q}{p}}  \cdot  \biggl( \displaystyle\int_{B_{t'}}|Dv|^{\frac{np}{n-2 \mu}}\dd x \biggr)^{\frac{(n-2 \mu)qp'}{ p}}, \label{int6}
\end{align}
where $p' = \frac{q-p}{\mu (2q-p)}$.

Inserting \eqref{interpol} and \eqref{int6} in \eqref{estimate1.}, and exploiting the bounds 
\begin{gather}\label{ineq.1}
\dfrac{n(q-p)}{\mu p} <1,  \quad \dfrac{nq(q-p)}{\mu  p(2q-p)} <1,
\end{gather}
which hold by assumption \eqref{gap} and for $\mu \in (\frac{n(q-p)}{p}, \alpha)$, from Young's inequality, we infer
\begin{align}\nonumber
 & \nu  \displaystyle\int_{\Omega}  \eta^{2} (|\tau_{h}V_p(Dv)|^{2} 
+  a(x+h)|\tau_{h}V_p(Dv)|^{2}) \dd x \notag\\
\leq & 3\varepsilon \displaystyle\int_{\Omega} \eta^2 |\tau_hV_p(Dv)|^2 \dd x \notag\\
&+ C_{\varepsilon}(\tilde{L},p,q, \Vert a \Vert_{\infty})|h|^{2\gamma}[D \psi]^{2q-p}_{B^{\gamma}_{2q-p, \infty}(B_{r})}\notag\\
&+ C_{\varepsilon, \theta}(\tilde{L},n,p,q, \Vert a \Vert_{\infty})|h|^{2 \gamma} \biggl( \displaystyle\int_{B_r}(1+|Dv|)^{p}\dd x \biggr)^{\frac{\delta(2q-p)\tilde{p}}{p}} \notag\\
&+ \theta |h|^{2 \gamma}\biggl( \displaystyle\int_{B_{t'}}(1+|Dv|)^{\frac{np}{n-2 \mu}}\dd x \biggr)^{\frac{n-2 \mu}{n}}\notag\\
&+\dfrac{C_{\varepsilon}(\tilde{L},n,p,q, \Vert a \Vert_{\infty})}{(t-s)^2}|h|^2  \displaystyle\int_{B_r}|D \psi|^{2q-p}\dd x  \notag\\
&+ \theta |h|^2 \biggl( \displaystyle\int_{B_{t'}}(1+|Dv|)^{\frac{np}{n-2 \mu}}\dd x \biggr)^{\frac{n-2 \mu}{n}} +\dfrac{C_{\varepsilon, \theta}(\tilde{L},n,p,q, \Vert a \Vert_{\infty})}{(t-s)^{2\tilde{p}}} |h|^2 \biggl( \displaystyle\int_{B_r}(1+|Dv|)^{p}\dd x \biggr)^{\frac{\tilde{p}\delta(2q-p)}{p}} \notag\\
&+ C_{\varepsilon, \theta}|h|^{2 \alpha} \biggl( \displaystyle\int_{B_r}|Dv|^{p}\dd x \biggr)^{\frac{\tilde{p}\delta (2q-p)}{p}}+ \theta |h|^{2 \alpha} \biggl( \displaystyle\int_{B_{t}}|Dv|^{\frac{np}{n-2 \mu}}\dd x \biggr)^{\frac{n-2 \mu}{n}}\notag\\
&+ C_{\theta}(q)|h|^{\alpha+\gamma}  [D \psi]^{q}_{B^{\gamma}_{2q-p, \infty}(B_{r})}+ C_{\theta}(n,p,q)|h|^{\alpha+\gamma} \left( \displaystyle\int_{B_r}|Dv|^p \dd x \right)^\frac{p^*\delta q}{p} \notag\\
&+\theta |h|^{\alpha+\gamma}\left( \ibt|Dv|^\frac{np}{n-2\mu} \dd x \right)^\frac{n-2\mu}{n} \notag\\
&+ \frac{C_\theta(n,p,q)}{t-s}|h|^{\alpha +1} \left(\displaystyle\int_{B_r} |D\psi|^{2q-p} \dd x\right)^\frac{q}{2q-p}  \nonumber \\
&+\frac{C_\theta(n,p,q)}{(t-s)^{p^*}} |h|^{\alpha+1}  \left( \displaystyle\int_{B_r}|Dv|^{p} \dd x\right)^{\frac{p^* \delta q}{p}} +\theta |h|^{\alpha+1} \left( \ibtt |Dv|^\frac{np}{n-2\mu} \dd x \right)^\frac{n-2\mu}{n}.\label{exp}
\end{align}
for some constant $\theta \in (0,1)$, where we set $\tilde{p}= \frac{\mu p }{\mu p -n(q-p)}$, $p^* = \frac{\mu p(2q-p)}{\mu p(2q-p)  -n(q-p)q}$.

For a better readability we now define
\begin{align*}
A=  & C_{\varepsilon}(\tilde{L},p,q, \Vert a \Vert_{\infty})[D \psi]^{2q-p}_{B^{\gamma}_{2q-p, \infty}(B_{r})}+ C_{\varepsilon, \theta}(\tilde{L},n,p,q, \Vert a \Vert_{\infty})\biggl( \displaystyle\int_{B_r} (1+|Dv|)^{p}\dd x \biggr)^{\frac{\delta(2q-p)\tilde{p}}{p}} \notag\\
&+ C_{\varepsilon, \theta}  \biggl( \displaystyle\int_{B_r}|Dv|^{p}\dd x \biggr)^{\frac{\tilde{p}\delta (2q-p)}{p}} \notag \\
&+ C_{\theta}(q)|h|^{\alpha+\gamma}  [D \psi]^{q}_{B^{\gamma}_{2q-p, \infty}(B_{r})}+ C_{\theta}(n,p,q)|h|^{\alpha+\gamma} \left( \displaystyle\int_{B_r}|Dv|^p \dd x \right)^\frac{p^*\delta q}{p}  \notag\\
B_1=& C_{\varepsilon}(\tilde{L},n,p,q, \Vert a \Vert_{\infty}) \displaystyle\int_{B_r}|D \psi|^{2q-p}\dd x ,\\
B_2= &C_{\varepsilon, \theta}(\tilde{L},n,p,q, \Vert a \Vert_{\infty})  \biggl( \displaystyle\int_{B_r}(1+|Dv|)^{p}\dd x \biggr)^{\frac{\tilde{p}\delta(2q-p)}{p}} ,\\
B_3 = &C_\theta(n,p,q) \left(\displaystyle\int_{B_r} |D\psi|^{2q-p} \dd x\right)^\frac{q}{2q-p} ,  \\
B_4 = & C_\theta(n,p,q)  \left( \displaystyle\int_{B_r} |Dv|^{p} \dd x\right)^{\frac{p^* \delta q}{p}}, 
\end{align*}
so that we can rewrite the previous estimate as
\begin{align}
&\nu  \displaystyle\int_{\Omega}  \eta^{2} (|\tau_{h}V_p(Dv)|^{2}  
+  a(x+h)|\tau_{h}V_p(Dv)|^{2}) \dd x \notag\\
\leq & 3\varepsilon \displaystyle\int_{\Omega} \eta^2 |\tau_hV_p(Dv)|^2 \dd x\notag\\
 &+  \theta (|h|^{2 \alpha}+|h|^{\alpha + \gamma} ) \left( \ibt (1+|Dv|)^\frac{np}{n-2\mu} \dd x \right)^\frac{n-2\mu}{n} \notag\\
 &+ \theta (|h|^2+|h|^{2 \gamma}+|h|^{\alpha +1} ) \left( \ibtt (1+|Dv|)^\frac{np}{n-2\mu} \dd x \right)^\frac{n-2\mu}{n} \notag\\
 &+ (|h|^{2 \gamma}+|h|^{2 \alpha}+|h|^{\alpha + \gamma})A + |h|^2\dfrac{B_1}{(t-s)^2}+ |h|^2\dfrac{B_2}{(t-s)^{2\tilde{p}}} \notag\\
 & + |h|^{\alpha+1}\dfrac{B_3}{(t-s)^{p^{''}}}+ |h|^{\alpha +1}\dfrac{B_4}{(t-s)^{p^*}} .\notag
\end{align}
Choosing $\varepsilon= \frac{\nu}{6}$, we can reabsorb the first integral in the right hand side of the previous estimate by the left hand side, thus getting
\begin{align}
 & \displaystyle\int_{\Omega}  \eta^{2}  (|\tau_{h}V_p(Dv)|^{2}+a(x+h)|\tau_{h}V_q(Dv)|^{2}) \dd x \notag\\
 \leq & 3 \theta |h|^{2 \alpha} \left( \ibt (1+|Dv|)^\frac{np}{n-2\mu} \dd x \right)^\frac{n-2\mu}{n} + 3\theta |h|^{2 \alpha} \left( \ibtt (1+|Dv|)^\frac{np}{n-2\mu} \dd x \right)^\frac{n-2\mu}{n} \notag\\
 &+ |h|^{2 \alpha}A + |h|^2\dfrac{B_1}{(t-s)^2}+ |h|^2\dfrac{B_2}{(t-s)^{2\tilde{p}}}  + |h|^{2\alpha}\dfrac{B_3}{t-s}+ |h|^{2\alpha }\dfrac{B_4}{(t-s)^{p^*}} , \label{formula}
\end{align}
where we used the fact that $\alpha < \gamma$. 

Since the right hand side of \eqref{formula} depends on the integrability of $Dv$, in order to exploit inequality \eqref{diffquotient}, we need to derive an a priori estimate for the gradient of the minimzer $v$. First, we bound \eqref{formula} from below as follows
\begin{align}\label{tauVp}
& \displaystyle\int_{B_s} |\tau_h V_p(Dv)|^2 \dd x \notag\\
 \leq & \displaystyle\int_{B_s} (|\tau_{h}V_p(Dv)|^{2}+a(x+h)|\tau_{h}V_q(Dv)|^{2})\dd x \notag\\
 \leq  &|h|^{2 \alpha} \biggl\{  2 \theta  \left( \ibt (1+|Dv|)^\frac{np}{n-2\mu} \dd x \right)^\frac{n-2\mu}{n} + 3\theta  \left( \ibtt (1+|Dv|)^\frac{np}{n-2\mu} \dd x \right)^\frac{n-2\mu}{n}  \notag\\
 &+ A + \dfrac{B_1}{(t-s)^2}+ \dfrac{B_2}{(t-s)^{2\tilde{p}}}  + \dfrac{B_3}{t-s}+ \dfrac{B_4}{(t-s)^{p^*}} \biggr\}, 
\end{align} 
where we also used that $\eta = 1$ on $B_s$. Then, Lemma \ref{besovembed} and equality \eqref{Vp} imply
\begin{align}
\biggl( \displaystyle\int_{B_s}  |Dv|^{\frac{np}{n- 2 \mu }} \dd x \biggr)^{\frac{n- 2 \mu}{n}} \leq & 2 \theta  \left( \ibt (1+|Dv|)^\frac{np}{n-2\mu} \dd x \right)^\frac{n-2\mu}{n} + 3\theta  \left( \ibtt (1+|Dv|)^\frac{np}{n-2\mu} \dd x \right)^\frac{n-2\mu}{n} \notag\\
 &+ A + \dfrac{B_1}{(t-s)^2}+ \dfrac{B_2}{(t-s)^{2\tilde{p}}}  + \dfrac{B_3}{t-s}+ \dfrac{B_4}{(t-s)^{p^*}} ,\label{gradest}
\end{align}
for all $\mu \in (\frac{n(q-p)}{p}, \alpha )$.\\
Now, applaying the iteration Lemma \ref{lm2} twice, we obtain
\begin{gather}\label{apriorigrad}
\biggl( \displaystyle\int_{B_{r/4}}  |Dv|^{\frac{np}{n- 2 \mu }} \dd x \biggr)^{\frac{n- 2 \mu}{n}} \leq  C |h|^{2\alpha} \biggl\{ \dfrac{1}{r^{2 \tilde{p}}} \biggl(\displaystyle\int_{B_r}  (1+|Dv|^p+|D\psi|^{2q-p})\dd x \biggr)^\kappa+ [ D \psi ]^{2q-p}_{B^{\gamma}_{2q-p, \infty}(B_r)} \biggl\},
\end{gather}
thus, using Lemma \ref{besovembed}, from inequalities \eqref{apriorigrad} and \eqref{tauVp}, we deduce the a priori estimate
\begin{align}
&  \displaystyle\int_{B_{r/4}}(|\tau_{h}V_p(Dv)|^{2}+a(x+h)|\tau_{h}V_q(Dv)|^{2})\dd x \notag\\
 \leq & C |h|^{2\alpha} \biggl\{ \dfrac{1}{r^{2 \tilde{p}}} \biggl(\displaystyle\int_{B_r}  (1+|Dv|^p+|D\psi|^{2q-p})\dd x \biggr)^\kappa+ [ D \psi ]^{2q-p}_{B^{\gamma}_{2q-p, \infty}(B_r)} \biggl\}, \label{besovv}
\end{align}
for some $\mu < \alpha$, where $C := C(n,p,q, \mu, \Vert a \Vert_{\infty})$ and $\kappa:=  \frac{\delta(2q-p)\tilde{p}}{p} < \tilde{p}$.
\endproof

According to the previous result, we state the following remarks, which will be crucial for the proof of Theorem \ref{thmbes}. 
\begin{rmk}
From Proposition \ref{rappincr} $(iii)$,  it follows that
\begin{align}
|\tau_h(\sqrt{a(x)}V_q(Dv))|^2 
\leq  C a(x+h)|\tau_hV_q(Dv)|^2+C|V_q(Dv)|^2|\tau_ha(x)|. \label{tauH}
\end{align}
Combining \eqref{besovv} and \eqref{tauH}, we obtain
\begin{align}
& \displaystyle\int_{B_{r/4}} |\tau_h(\sqrt{a(x)}V_q(Dv))|^2 \dd x \notag\\
\leq &  \displaystyle\int_{B_{r/4}}(|\tau_hV_p(Dv)|^2+|\tau_h(\sqrt{a(x)}V_q(Dv))|^2\dd x \notag\\
\leq & C \displaystyle\int_{B_{r/4}}(|\tau_hV_p(Dv)|^2+a(x+h)|\tau_hV_q(Dv)|^2+|V_q(Dv)|^2|\tau_ha(x)|)\dd x \notag\\
 \leq & C |h|^{2\alpha} \biggl\{ \dfrac{1}{r^{2 \tilde{p}}} \biggl(\displaystyle\int_{B_r}  (1+|Dv|^p+|D\psi|^{2q-p})\dd x \biggr)^\kappa+ [ D \psi ]^{2q-p}_{B^{\gamma}_{2q-p, \infty}(B_r)} \biggl\}+C|h|^\alpha \Vert Dv \Vert^q_{L^{q}(B_r)}\notag\\
\leq & C |h|^{\alpha} \biggl\{ \dfrac{1}{r^{2 \tilde{p}}} \biggl(\displaystyle\int_{B_r}  (1+|Dv|^p+|D\psi|^{2q-p})\dd x \biggr)^\kappa+ \Vert Dv \Vert^q_{L^{q}(B_r)}+[ D \psi ]^{2q-p}_{B^{\gamma}_{2q-p, \infty}(B_r)} \biggl\}, \label{gradvq}
\end{align}
which is finite by Theorem \ref{highcomp}. Therefore,
$$\sqrt{a(x)}V_q(Dv) \in B^{\frac{\alpha}{2}}_{2,\infty,\text{loc}}(B).$$
Lemma \eqref{besovembed} yields
$$a(x)|Dv|^q \in L^{\frac{n}{n- 2 \beta}}_\text{loc}(B), \quad \forall \beta < \dfrac{\alpha}{2}.$$
\end{rmk}

\begin{rmk}\label{remark}
Choosing $\mu < \alpha$ s.t. $q= \frac{np}{n- 2 \mu}$, estimates \eqref{apriorigrad} and \eqref{gradvq} yield
\begin{align}
&  \displaystyle\int_{B_{r/4}}(|\tau_hV_p(Dv)|^2+|\tau_h(\sqrt{a(x)}V_q(Dv))|^2\dd x \notag\\
\leq & C |h|^{\alpha} \biggl\{ \dfrac{1}{r^{2 \tilde{p}}} \biggl(\displaystyle\int_{B_r}  (1+|Dv|^p+|D\psi|^{2q-p})\dd x \biggr)^\kappa+ [ D \psi ]^{2q-p}_{B^{\gamma}_{2q-p, \infty}(B_r)} \biggl\}^{\frac{q}{p}} \notag\\
\leq & C |h|^{\alpha} \biggl\{ \dfrac{1}{r^{2\tilde{p}_1} }\biggl(\displaystyle\int_{B_r}  (1+|Dv|^p+|D\psi|^{2q-p})\dd x \biggr)^{\kappa_1}+ [ D \psi ]^{q_1}_{B^{\gamma}_{2q-p, \infty}(B_r)}+1 \biggl\},\label{comparisonbesov}
\end{align}
where $\frac{ \tilde{p}q}{p}=\tilde{p}_1>1$, $\frac{\kappa q}{p}=\kappa_1< \tilde{p}_1 $ and $\frac{(2q-p) q}{p}=q_1<\tilde{p}_1 $, with $\tilde{p}$ and $\kappa$ introduced in \eqref{exp} and \eqref{besovv} respectively.
\end{rmk}

\section{Comparison}\label{seccomp}
In this section we prove a comparison Lemma (see Lemma \ref{comparison} below), where we estimate the distance between the solution $u$ to the problem \eqref{obpro} and the solution $v$ to the problem \eqref{frozen}. In order to do so, we first need the following lemma.
\begin{lem}
Let $\tilde{F}: \Omega \times \mathbb{R}^n \rightarrow \mathbb{R}$ be the function defined in \eqref{froint} under Assumption 1. Then there exists a positive constant $c=c(r,n,\nu)$ such that the following inequality holds for every $x \in \Omega $ and every $z_1,z_2 \in \mathbb{R}^n$
\begin{align}
c(|V_p(z_1)-V_p(z_2)|^2+ a(x)&|V_q(z_1)-V_q(z_2)|^2 ) \notag\\
& \leq \tilde{F}(x,z_1)-\tilde{F}(x,z_2)- \langle D_{\xi}\tilde{F}(x,z_2)  ,z_1 -z_2 \rangle. \label{stimacomp}
\end{align}
\end{lem}
\proof
We start proving that for every $r \geq 2$ there exists a constant $c=c(r,n)$ such that
\begin{align}
c(r,n) |V_r(z_1)-V_r(z_2)|^2 \leq  g_r(z_1)-g_r(z_2)- \langle D_\xi g_r(z_2), z_1 -z_2 \rangle , 
\end{align}
where we denote $g_r(z):=|z|^r$.

Let us consider the function $G_r: [0,1] \rightarrow \mathbb{R}$ defined by $G_r(t):= g_r(t z_1 + (1-t) z_2)$. Since $G_r \in \mathcal{C}^2([0,1])$, by using Taylor's formula with integral remainder, we obtain
\begin{equation}
G_r(1)=G_r(0)+G'_r(0)+ \displaystyle\int_0^1 (1-s)G''_r(s) ds. \label{Taylor}
\end{equation}
Since 
\begin{align*}
G'_r(t)=& \langle D_\xi g_r(t z_1 +(1-t)z_2), z_1 -z_2 \rangle ,\\
G''_r(t) =& \langle D_{\xi \xi} g_r(t z_1 +(1-t)z_2)(z_1 -z_2), z_1 -z_2 \rangle,
\end{align*}
from \eqref{Taylor} we get
\begin{align}
g_r(z_1)-g_r(z_2)- \langle & D_\xi g_r(z_2),  z_1 -z_2 \rangle \notag\\
=& \displaystyle\int_0^1 (1-s) \langle D_{\xi \xi} g_r(s z_1 +(1-s)z_2)(z_1 -z_2), z_1 -z_2 \rangle ds \notag\\
\geq & c(r) |z_1 -z_2 |^2 \displaystyle\int_0^1 (1-s)  |s z_1 +(1-s)z_2|^{r-2} ds. \label{taylorest}
\end{align}
Now, we want to estimate from below $|s z_1 +(1-s)z_2|^{r-2} $. If $|z_1| \leq |z_2|$ and $s \in [3/4,1]$, then $-1/4 \leq s-1 \leq 0$ and
$$|s z_1 +(1-s)z_2| \geq s |z_1| +(s-1)|z_2| \geq \dfrac{3}{4} |z_1| - \dfrac{1}{4} |z_2| \geq \dfrac{1}{4} (|z_1|+|z_2|),$$
while, if $|z_2| > |z_2|$ and $s \in [0,1/4]$, then $3/4 \leq 1-s \leq 1$ and
$$|s z_1 +(1-s)z_2| \geq -s |z_1| +(1-s)|z_2| \geq -\dfrac{1}{4} |z_1| + \dfrac{3}{4} |z_2| \geq \dfrac{1}{4} (|z_1|+|z_2|).$$
Therefore
\begin{equation}
|s z_1 +(1-s)z_2|^{r-2} \geq 4^{2-r} (|z_1|+|z_2|)^{r-2} \label{eq}
\end{equation}
holds on a suitable subinterval of $[0,1]$. Eventually, inserting \eqref{eq} in \eqref{taylorest} we obtain
\begin{align*}
g_r(z_1)-g_r(z_2)- \langle  D_\xi g_r(z_2),  z_1 -z_2 \rangle \geq& c(r) (|z_1|+|z_2|)^{r-2}|z_1 -z_2 |^2 \\
\geq & c(r,n) |V_r(z_1)-V_r(z_2)|^2,
\end{align*}
where in the last inequality we used Lemma \ref{D1}.

At this point, using the bound from below on $b$ in Assumption 1 and estimate \eqref{stimacomp} we deduce
\begin{align*}
\tilde{F}(x,z_1)-\tilde{F}(x,z_2)-& \langle D_{\xi}\tilde{F}(x,z_2)  ,z_1 -z_2 \rangle \\
=& b(x_0,u_B)[g_p(z_1)-g_p(z_2)- \langle D_{\xi}g_p(z_2)  ,z_1 -z_2 \rangle \\
&+a(x)(g_q(z_1)-g_q(z_2)- \langle D_{\xi}g_q(z_2)  ,z_1 -z_2 \rangle)] \\
\geq & c(r,n, \nu) (|V_p(z_1)-V_p(z_2)|^2+ a(x)|V_q(z_1)-V_q(z_2)|^2 ). 
\end{align*}
\endproof

\begin{rmk}In the proof of Lemma \ref{comparison} we will take advantage of the higher integrability results established in Section \ref{higher_int}. Indeed, the assumption $D \psi \in B^{\gamma}_{2q-p,\infty,\text{loc}}(\Omega)$ and Lemma \ref{besovembed} imply that $D \psi \in L^{\frac{n(2q-p)}{n-\mu (2q-p)}}$, for every $0 < \mu < \gamma$. Therefore, $H(x,D \psi)$ belong to some $L^{m}$, with $m>1 $. 
\end{rmk}

\begin{lem}\label{comparison}
Let $u$ be a solution to \eqref{obpro} and $v \in u+W^{1,p}_0(B)$ be the solution to \eqref{frozen}, under Assumptions 1 and 2, for exponents $2 \leq p < q$ verifying
$$\dfrac{q}{p}< 1 + \dfrac{\alpha}{n}.$$ 
If $$D \psi \in B^{\gamma}_{2q-p,\infty,\text{loc}}(\Omega),$$ for $0< \alpha < \gamma <1$, then
\begin{align}
\displaystyle\int_B |V_p(Du)-V_p(Dv)|^2+ a(x)&|V_q(Du)-V_q(Dv)|^2 \dd x \notag\\
&\leq CR^{\sigma} \displaystyle\int_{2B}(1+(H(x,Du))^m+(H(x,D\psi))^m)\dd x, \label{comparisonest}
\end{align}
with $\sigma= \min \{ \beta, m-1\}$, where $\beta $ is the exponent appearing in the Assumption 2 and where $m$ is the minimum of the two higher integrability exponents of Theorems \ref{higherint} and \ref{higint2}.
\end{lem}

\proof
Assumption 1, the definition of $\tilde{F}$ and the minimality of $v$ imply
\begin{align}\label{estimate1comp}
\displaystyle\int_B H(x,Dv)\dd x \leq C \displaystyle\int_B \tilde{F}(x,Dv)\dd x \leq \displaystyle\int_B\tilde{F}(x,Du)\dd x \leq C \displaystyle\int_B H(x,Du)\dd x,
\end{align}
on the other hand, Theorem \ref{higint2} yields
\begin{align}\label{integr}
\displaystyle\int_B (H(x,Dv))^m \dd x \leq \displaystyle\int_B [(H(x,Du))^m +(H(x,D\psi))^m]\dd x,
\end{align}
for some $m>1$. From inequality \eqref{stimacomp} we get
\begin{align*}
\displaystyle\int_B |V_p(Du)-V_p(Dv)|^2+ a(x)&|V_q(Du)-V_q(Dv)|^2 \dd x \\
&\leq C \displaystyle\int_B \tilde{F}(x,Du)-\tilde{F}(x,Dv)- \langle D_{\xi}\tilde{F}(x,Dv)  ,Du-Dv \rangle \dd x ,
\end{align*}
moreover, recalling inequality \eqref{varin}, i.e.
\begin{align}
\displaystyle\int_B \langle D_{\xi}H(x,Dv)  ,Du-Dv \rangle \dd x \geq 0,
\end{align}
and that $b(x_0,u_B)\geq \nu > 0$,
we deduce 
\begin{align*}
\displaystyle\int_B |V_p(Du)-V_p(Dv)|^2+ a(x)&|V_q(Du)-V_q(Dv)|^2 \dd x\leq C \displaystyle\int_B \tilde{F}(x,Du)-\tilde{F}(x,Dv)\dd x. 
\end{align*}
Hence, we can write the previous estimate as follows
\begin{align}
\displaystyle\int_B |V_p(Du)-V_p(Dv)&|^2+ a(x)|V_q(Du)-V_q(Dv)|^2 \dd x \notag\\
\leq& C \displaystyle\int_B \tilde{F}(x,Du)-\tilde{F}(x,Dv)\dd x \notag\\
=&C \displaystyle\int_B [b(x_0,u_B)H(x,Du)-b(x_0,u_B)H(x,Dv)]\dd x \notag\\
=&C \displaystyle\int_B [b(x_0,u_B)H(x,Du)-b(x,u_B)H(x,Du)]\dd x \notag\\
&+C \displaystyle\int_B [b(x,u_B)H(x,Du)-b(x,u)H(x,Du)]\dd x \notag\\
&+C \displaystyle\int_B [b(x,u)H(x,Du)-b(x,v)H(x,Dv)]\dd x \notag\\
&+C \displaystyle\int_B [b(x,v)H(x,Dv)-b(x,v_B)H(x,Dv)]\dd x \notag\\
&+C \displaystyle\int_B [b(x,v_B)H(x,Dv)-b(x,u_B)H(x,Dv)]\dd x \notag\\
&+C \displaystyle\int_B [b(x,u_B)H(x,Dv)-b(x_0,u_B)H(x,Dv)]\dd x \notag\\
=& C[I_1+I_2+I_3+I_4+I_5+I_6].\label{estcomp}
\end{align}
We proceed estimating the various pieces arising up from \eqref{estcomp}.

By Assumption 2 and estimate \eqref{estimate1comp}, we get
\begin{align}
I_1+I_6 \leq & \displaystyle\int_B \omega_b(|x-x_0|)H(x,Du)\dd x+\displaystyle\int_B \omega_b(|x-x_0|)H(x,Dv)\dd x \notag\\
\leq & \displaystyle\int_B |x-x_0|^\beta(H(x,Du)+H(x,Dv))\dd x \notag\\
\leq & CR^{\beta} \displaystyle\int_B H(x,Du)\dd x\notag\\
\leq &  CR^{\beta} \displaystyle\int_B [1+(H(x,Du))^m]\dd x.\label{I.1}
\end{align}
Now, we take care of the integral $I_2$. From Assumption 2, Young's and Poincaré's inequalities, we infer
\begin{align}
I_2 \leq & \displaystyle\int_B \omega_b(|u-u_B|)H(x,Du)\dd x \notag\\
= & \displaystyle\int_B \dfrac{1}{R^{\frac{\sigma}{1+\sigma}}}\omega_b(|u-u_B|) R^{\frac{\sigma}{1+\sigma}}H(x,Du)\dd x \notag\\
\leq & C\displaystyle\int_B \dfrac{1}{R}\omega_b(|u-u_B|)^{\frac{1+\sigma}{\sigma}}+ R^{\sigma}(H(x,Du))^{1+\sigma}\dd x \notag\\
\leq & CR^\sigma\displaystyle\int_B \dfrac{1}{R^{1+\sigma}}|u-u_B|^{1+\sigma}+ (H(x,Du))^{1+\sigma}\dd x \notag\\
\leq & CR^\sigma\displaystyle\int_B |Du|^{1+\sigma}+ (H(x,Du))^{1+\sigma}\dd x \notag\\
\leq & CR^\sigma\displaystyle\int_B (1+|Du|^{p(1+\sigma)})+ (H(x,Du))^{1+\sigma}\dd x \notag\\
\leq & CR^\sigma\displaystyle\int_B  [1+(H(x,Du))^m]\dd x ,\label{I.2}
\end{align}
where $\sigma := \min \{ \beta,m-1\}$.

The minimality of $u$ yields that
\begin{align}
I_3 \leq 0. \label{I.3}
\end{align}
Arguing analogously as for the integral $I_2$, we obtain
\begin{align}
I_4 \leq &  \displaystyle\int_B \omega_b(|v-v_B|)H(x,Dv)\dd x \notag\\
\leq & CR^\sigma\displaystyle\int_B \dfrac{1}{R^{1+\sigma}}|v-v_B|^{1+\sigma}+ (H(x,Dv))^{1+\sigma}\dd x \notag\\
\leq & CR^\sigma\displaystyle\int_B |Dv|^{1+\sigma}+ (H(x,Dv))^{1+\sigma}\dd x \notag\\
\leq & CR^\sigma\displaystyle\int_B  [1+(H(x,Dv))^m]\dd x \notag\\
 \leq & CR^\sigma\displaystyle\int_B  [1+(H(x,Du))^m+(H(x,D\psi))^m]\dd x,\label{I.4}
\end{align}
where in the last inequality we used \eqref{integr}.

Since $u=v $ on $\partial B$, using Poicaré inequality for the function $u-v$, we infer the following estimate for $I_5$.
\begin{align}
I_5 \leq & \displaystyle\int_B \omega_b(|u_B-v_B|)H(x,Dv)\dd x \notag\\
\leq & CR^\sigma\displaystyle\int_B \dfrac{1}{R^{1+\sigma}}\omega_b(|u_B-v_B|)^{\frac{1+\sigma}{\sigma}}+ (H(x,Du))^{1+\sigma}\dd x \notag\\
\leq & CR^\sigma\displaystyle\int_B \dfrac{1}{R^{1+\sigma}}|u-v|^{1+\sigma}+ (H(x,Dv))^{1+\sigma}\dd x \notag\\
\leq & CR^\sigma\displaystyle\int_B |Du|^{1+\sigma}+ |Dv|^{1+\sigma}+(H(x,Dv))^{1+\sigma}\dd x \notag\\
\leq & CR^\sigma\displaystyle\int_B  [1+(H(x,Du))^m+(H(x,Dv))^m]\dd x \notag\\
\leq & CR^\sigma\displaystyle\int_B  [1+(H(x,Du))^m+(H(x,D\psi))^m]\dd x, \label{I.5}
\end{align}
where in the inequality we used estimate \eqref{estimate1comp}. Finally, inserting estimates \eqref{I.1}--\eqref{I.5} in \eqref{estcomp}, we get the desired estimate.
\endproof

\section{Main result}\label{mainsec}
In order to prove Theorem \ref{mainthm} we follow the strategy first proposed in \cite{kristensen.mingione}.\\
Before proving our main result, in Section \ref{subsec6.1}, we fix some further notation and derive a preliminary regularity theorem for solutions to \eqref{obpro}.\\
For a ball $\mathcal{B} \Subset \Omega$ of radius $R$, we will denote by $\mathcal{Q}_1= \mathcal{Q}_1(\mathcal{B})$ and $\mathcal{Q}_2= \mathcal{Q}_2(\mathcal{B})$ the largest and the smallest cubes, concentric to $\mathcal{B}$ and with sides parallel to the coordinate axes, contained in $\mathcal{B}$ and containing $\mathcal{B}$ respectively. It is easy to verify that $|\mathcal{B}|\approx |\mathcal{Q}_1|\approx |\mathcal{Q}_2|\approx R^n$. We also denote the enlarged ball by $\hat{\mathcal{B}}=4\mathcal{B}$. We set
$$ \mathcal{Q}_1= \mathcal{Q}_1(\mathcal{B}) \ \ \ \hat{\mathcal{Q}}_2=\mathcal{Q}_2(\hat{\mathcal{B}})$$
so that we have the following chain of inclusions
$$ \mathcal{Q}_1 \subset \mathcal{B} \Subset 2 \mathcal{B} \Subset \mathcal{Q}_1(\hat{\mathcal{B}}) \subset {\hat{\mathcal{B}}} \subset \hat{\mathcal{Q}}_2  .$$
In what follows, we shall always take $\mathcal{B}$ such that $\mathcal{Q}_2(\hat{\mathcal{B}}) \Subset \Omega $.

Our next result shows that a fractional differentiability property on the gradient of the obstacle transfers to a higher fractional differentiability for the gradient of the minimizer.
\begin{thm}\label{thmbes}
Let $u$ be a solution to \eqref{obpro} under Assumptions 1 and 2, for exponents $2 \leq p < \frac{n}{\alpha}$, $p<q$ verifying
$$\dfrac{q}{p}< 1 + \dfrac{\alpha}{n}.$$ Then the following implication
$$D\psi \in B^\gamma_{2q-p,\infty,\text{loc}}(\Omega) \Rightarrow V_p(Du), \ \sqrt{a(x)}V_q(Du) \in B^{\sigma_\alpha}_{2, \infty, \text{loc}}(\Omega) $$
holds provided $0 < \alpha < \gamma <1$, with $\sigma_\alpha = \sigma_\alpha(p,q,n,\alpha,\beta, m)$, where $\beta $ is the exponent appearing in the Assumption 2 and where $m$ is the minimum of the two higher integrability exponents of Theorems \ref{higherint} and \ref{higint2}.
\end{thm}
\proof 
Let us fix arbitrary open subsets $\Omega'  \Subset \Omega''  \Subset \Omega$ and choose $x_0 \in \Omega' $. We recall the definition of $\tilde{p}_1$ from \eqref{comparisonbesov}. Let $\delta \in \left(0,\frac{\alpha}{2\tilde{p}_1}\right)$ be chosen later and consider the ball $\mathcal{B}=\mathcal{B}(x_0,|h|^{\delta})$ with $|h|$ sufficiently small, depending on the dimension $n$, the parameter $\delta$ and the distance between $\Omega'$ and the boundary of $\Omega'' $ such that $\hat{\mathcal{Q}}_2 \Subset \Omega''$. 
Furthermore, let $v \in u+W^{1,p}_0(B)$ be the solution to \eqref{frozen} with $B=\hat{\mathcal{B}}$.

We estimate the difference quotient for $V_p(Du)$ and $\sqrt{a(x)}V_q(Du)$ as follows
\begin{align}
&\displaystyle\int_{\mathcal{B}} |\tau_hV_p(Du)|^2+|\tau_h(\sqrt{a(x)}V_q(Du))|^2 \dd x\notag\\
= &  \displaystyle\int_{\mathcal{B}} |V_p(Du(x+h))-V_p(Du(x))|^2 \dd x \notag\\
&+ \displaystyle\int_{\mathcal{B}} |\sqrt{a(x+h)}V_q(Du(x+h))-\sqrt{a(x)}V_q(Du(x))|^2 \dd x \notag\\
\leq & C \displaystyle\int_{\mathcal{B}} |V_p(Du(x+h))-V_p(Dv(x+h))|^2 \dd x \notag\\
&+ C\displaystyle\int_{\mathcal{B}} |V_p(Dv(x+h))-V_p(Dv(x))|^2 \dd x \notag\\
&+ C\displaystyle\int_{\mathcal{B}} |V_p(Dv(x))-V_p(Du(x))|^2 \dd x\notag\\
&+ C \displaystyle\int_{\mathcal{B}} |\sqrt{a(x+h)}V_q(Du(x+h))-\sqrt{a(x+h)}V_q(Dv(x+h))|^2 \dd x \notag\\
&+ C\displaystyle\int_{\mathcal{B}} |\sqrt{a(x+h)}V_q(Dv(x+h))-\sqrt{a(x)}V_q(Dv(x))|^2 \dd x \notag\\
&+ C\displaystyle\int_{\mathcal{B}} |\sqrt{a(x)}V_q(Dv(x))-\sqrt{a(x)}V_q(Du(x))|^2 \dd x. \label{1.0.1}
\end{align}

Notice that if $x \in \mathcal{B} $, then $x+h \in \hat{\mathcal{B}}$, for $|h| \leq 1$. Thus, we get
\begin{align}
& \displaystyle\int_{\mathcal{B}} |V_p(Du(x+h))-V_p(Dv(x+h))|^2 \dd x  \notag\\
&+ \displaystyle\int_{\mathcal{B}} |\sqrt{a(x+h)}V_q(Du(x+h))-\sqrt{a(x+h)}V_q(Dv(x+h))|^2 \dd x \notag\\
\leq & \displaystyle\int_{\hat{\mathcal{B}}} |V_p(Du)-V_p(Dv)|^2 +a(x)|V_q(Du)-V_q(Dv)|^2\dd x. \label{1.0.2}
\end{align}
Inserting inequality \eqref{1.0.2} in \eqref{1.0.1}, we obtain
\begin{align}
&\displaystyle\int_{\mathcal{B}} |\tau_hV_p(Du)|^2+|\tau_h(\sqrt{a(x)}V_q(Du))|^2 \dd x\notag\\
\leq & C \displaystyle\int_{\hat{\mathcal{B}}} |V_p(Du)-V_p(Dv)|^2 +a(x)|V_q(Du)-V_q(Dv)|^2\dd x \notag\\
&+ C\displaystyle\int_{\mathcal{B}} |\tau_hV_p(Dv)|^2 +  |\tau_h(\sqrt{a(x)}V_q(Dv))|^2 \dd x \notag\\
=:& J_1+J_2. \label{1.0}
\end{align}

From estimate \eqref{comparisonest} applied over the ball $\hat{\mathcal{B}}$, we infer
\begin{align}
J_1\leq C|h|^{\sigma \delta} \displaystyle\int_{\hat{\mathcal{Q}}_2}(1+(H(x,Du))^m+(H(x,D\psi))^m)\dd x, \label{1.1}
\end{align}
where we used that the radius of $\hat{\mathcal{B}}$ is proportional to $|h|^\delta$. Now estimate \eqref{comparisonbesov} (see Remark \ref{remark}) applied over the ball $\mathcal{B}$ yields

\begin{gather}\label{1.2}
J_2 \leq C |h|^{\alpha-2 \delta \tilde{p}_1} \biggl(  \displaystyle\int_{\hat{\mathcal{B}}}  (1+|Dv|^p+|D\psi|^{2q-p})\dd x\biggr)^{\kappa_1}+  C |h|^{\alpha} [ D \psi ]^{q_1}_{B^{\gamma}_{2q-p, \infty}(\hat{\mathcal{B}})}+ C|h|^\alpha ,
\end{gather}
recalling that the radius of $\mathcal{B}$ is $|h|^\delta$. Inserting \eqref{1.1} and \eqref{1.2} in \eqref{1.0}, we get

\begin{align}
&\displaystyle\int_{\mathcal{B}} |\tau_hV_p(Du)|^2+|\tau_h(\sqrt{a(x)}V_q(Du))|^2 \dd x\notag\\
 \leq & C|h|^{\sigma \delta} \displaystyle\int_{\hat{\mathcal{Q}}_2}(1+(H(x,Du))^m+(H(x,D\psi))^m)\dd x \notag\\
&+ C |h|^{\alpha-2 \delta \tilde{p}_1} \biggl(  \displaystyle\int_{\hat{\mathcal{B}}}  (1+|Dv|^p+|D\psi|^{2q-p})\dd x\biggr)^{\kappa_1}+  C |h|^{\alpha} [ D \psi ]^{q_1}_{B^{\gamma}_{2q-p, \infty}(\hat{\mathcal{B}})}+ C|h|^\alpha \notag\\
\leq &  C|h|^{\sigma \delta} \displaystyle\int_{\hat{\mathcal{Q}}_2}(1+(H(x,Du))^m+(H(x,D\psi))^m)\dd x \notag\\
&+ C |h|^{\alpha-2 \delta \tilde{p}_1} \biggl(  \displaystyle\int_{\hat{\mathcal{Q}}_2}  (1+H(x,Du))\dd x\biggr)^{\kappa_1}+  C |h|^{\alpha} [D \psi ]^{q_1}_{B^{\gamma}_{2q-p, \infty}(\hat{\mathcal{Q}}_2)} \notag\\
&+ C |h|^{\alpha-2 \delta \tilde{p}_1} \biggl(  \displaystyle\int_{\hat{\mathcal{Q}}_2} |D\psi|^{2q-p}\dd x\biggr)^{\kappa_1}+C|h|^\alpha, \label{1.3}
\end{align}
where in the last inequality we used \eqref{estimate1comp}. 

Now we choose $\delta$ in order to minimize the right hand side of the previous estimate. It is easy to check that the best possible estimate is given by the choice 

$$\delta =\dfrac{ \alpha}{\sigma +2 \tilde{p}_1} \in \biggl(0, \dfrac{\alpha}{2\tilde{p}_1} \biggr).$$

With such a choice of $\delta$ estimate \eqref{1.3} becomes

\begin{align}
&\displaystyle\int_{\mathcal{B}} |\tau_hV_p(Du)|^2+|\tau_h(\sqrt{a(x)}V_q(Du))|^2 \dd x\notag\\
\leq &  C|h|^{\frac{\alpha \sigma}{\sigma + 2 \tilde{p}_1}} \biggl\{ \displaystyle\int_{\hat{\mathcal{Q}}_2}(1+(H(x,Du))^m+(H(x,D\psi))^m)\dd x + \Vert D \psi \Vert_{B^{\gamma}_{2q-p, \infty}(\hat{\mathcal{Q}}_2)} +1 \biggr\}^{\kappa_2}, \label{1.4}
\end{align}
where $\kappa_2 := \kappa_2(n,p,q,\mu)$, for some $\mu < \alpha$.

At this point, arguing as in \cite[Lemma 4.5]{kristensen.mingione}, a covering argument allows us to replace the cubes $\mathcal{Q}_1$ and $\hat{\mathcal{Q}}_2$ with the fixed open subsets $\Omega'$ and $\Omega''$, respectively. Indeed for each $|h|\in\mathbb{R}^n$ sufficiently small we can find balls $\mathcal{B}_1 = \mathcal{B}_1(x_1,|h|^{\sigma}),...,\mathcal{B}_K = \mathcal{B}_K(x_K,|h|^{\sigma})$, being $K=K(h)\in \mathbb{N}$, such that the corresponding inner cubes $\mathcal{Q}_1(\mathcal{B}_1),...,\mathcal{Q}_1(\mathcal{B}_K)$ are disjoint and satisfy
$$ \left| \Omega' \setminus \displaystyle\bigcup_{k=1}^{K}\mathcal{Q}_1(\mathcal{B}_k) \right| =0 . $$
By our assumption we have that $\mathcal{Q}_2(\hat{\mathcal{B}}_k) \subset \Omega''$, for every $k \leq K$ and each of the dilated outer cubes $\mathcal{Q}_2(\hat{\mathcal{B}}_k)$ intersects at most ($16 \sqrt{n}$) of the other cubes $\mathcal{Q}_2(\hat{\mathcal{B}}_j)$, with $j \neq k$. Hence, after summing up \eqref{1.4} over the inner cubes $\mathcal{Q}_1 \in \{ \mathcal{Q}_1(\mathcal{B}_1),...,\mathcal{Q}_1(\mathcal{B}_K) \}$, and enlarging the constant by a fixed factor only depending on $n$ and $p$ (in particular independent of $h$), we arrive at

\begin{align}
&\displaystyle\int_{\Omega'} |\tau_hV_p(Du)|^2+|\tau_h(\sqrt{a(x)}V_q(Du))|^2 \dd x\notag\\
\leq &  C|h|^{\frac{\alpha \sigma}{\sigma + 2 \tilde{p}_1}} \biggl\{ \displaystyle\int_{\Omega''}(1+(H(x,Du))^m+(H(x,D\psi))^m)\dd x + \Vert D \psi \Vert_{B^{\gamma}_{2q-p, \infty}(\Omega'')} +1 \biggr\}^{\kappa_2}. \label{1.5}
\end{align}

Since the right hand side of the previous estimate is finite by our assumptions, it follows by arbitrariness of $\Omega'$ that 

$$V_p(Du), \ \sqrt{a(x)}V_q(Du) \in B^{\frac{\alpha \sigma}{2(\sigma + 2 \tilde{p}_1)}}_{2, \infty}(\Omega) \quad \text{locally}.$$
Setting 
\begin{equation}\label{sa}
\sigma_\alpha := \frac{\alpha \sigma}{2(\sigma + 2 \tilde{p}_1)},
\end{equation} 
it follows the conclusion.
\endproof

\subsection{Proof of Theorem \ref{mainthm}}\label{subsec6.1}
\noindent We are now able to give the proof of the main result of this work. \smallskip\\

\noindent Let us consider the function
\begin{equation}\label{funA}
A(t)= \dfrac{\alpha \sigma}{2 [2 (\tilde{p}_1- \kappa_1 t)+ \sigma]}, \qquad  \ \forall t \in \biggl(  0, \dfrac{\sigma +2 \tilde{p}_1 - \sqrt{(\sigma +2 \tilde{p}_1)^2-4 \kappa_1  \alpha \sigma}}{ 4\kappa_1} \biggr)=:(0, \tilde{\sigma}),
\end{equation}
where $\tilde{p}_1, \kappa_1$ are defined in \eqref{comparisonbesov}, $\sigma$ is defined in Lemma \ref{comparison} and $\alpha$ is the exponent appearing in Assumption $1$.

It is easy to see that $t \mapsto A(t)$ is increasing and that
\begin{align}
t < A(t) & < \tilde{\sigma} , \label{A1} \\
A ( \tilde{\sigma}) & = \tilde{\sigma}. \label{A2}
\end{align}
It is worth noticing that
\begin{equation}\label{sigma}
\sigma_\alpha < \tilde{\sigma}< \dfrac{\alpha \sigma}{2\tilde{p}_1},
\end{equation}
where $\sigma_\alpha$ was introduced in \eqref{sa}.
Indeed, owing to \eqref{sa}, the first part of inequality \eqref{sigma} holds if, and only if,
$$ (\sigma+2 \tilde{p}_1)\sqrt{(\sigma +2 \tilde{p}_1)^2-4 \kappa_1  \alpha \sigma} <  (\sigma+2 \tilde{p}_1)^2 - 2 \kappa_1 \alpha \sigma. $$
The last inequality is satisfied if, and only if,
$$  (\sigma+2 \tilde{p}_1)^4 -4 \alpha \kappa_1 \sigma (\sigma+2 \tilde{p}_1)^2 < (\sigma+2 \tilde{p}_1)^4+4 \kappa_1^2\alpha^2\sigma^2 -4 \kappa_1 \alpha \sigma (\sigma+2 \tilde{p}_1)^2,$$
that is equivalent to
$$0 < 4 \kappa_1^2\alpha^2\sigma^2.$$
On the other hand, the second part of inequality \eqref{sigma} is valid if, and only if,
$$ \tilde{p}_1 (\sigma+2 \tilde{p}_1) -2 \kappa_1 \alpha \sigma < \tilde{p}_1 \sqrt{(\sigma +2 \tilde{p}_1)^2-4 \kappa_1  \alpha \sigma}, $$
or, equivalently,
$$\tilde{p}_1^2 (\sigma+2 \tilde{p}_1)^2+4\kappa_1^2\alpha^2 \sigma^2-4 \tilde{p}_1 \kappa_1  \alpha \sigma (\sigma+2 \tilde{p}_1) < \tilde{p}_1^2 (\sigma+2 \tilde{p}_1)^2-4 \tilde{p}_1^2 \kappa_1 \alpha \sigma.$$
The previous inequality can be written as
$$\kappa_1 \alpha \sigma - \tilde{p}_1 \sigma < \tilde{p}_1^2,$$
that holds true since $1 < \kappa_1 < \tilde{p}_1$ and $\alpha , \sigma \in (0,1)$.

Let us now fix
$$\theta_0 \in \biggl( 0,  \dfrac{\alpha \sigma}{2(\sigma +2 \tilde{p}_1)} \biggr) $$
and denote
$$\theta_j=A(\theta_{j-1}), \quad \forall j \in \mathbb{N}, \ j \geq 1.$$
Hence, the sequence $(\theta_j)_j$ is increasing and 
\begin{equation}
\displaystyle\lim_j \theta_j = \tilde{\sigma}. \label{theta}
\end{equation}
Now we define the sequence $( \iota_j )_j$ inductively as follows:
\begin{align*}
\iota_0 = & \dfrac{\theta_0}{2}+  \dfrac{\alpha \sigma}{4(\sigma +2 \tilde{p}_1)} < \dfrac{\alpha \sigma}{2(\sigma +2 \tilde{p}_1)} , \\
\iota_j= & \dfrac{\theta_j+A(\iota_{j-1})}{2} . 
\end{align*}
Using the fact that $A$ is increasing and \eqref{A1}, \eqref{A2}, we obtain
\begin{equation}
\theta_j < \iota_j < \tilde{\sigma}, \qquad \ \forall j \in \mathbb{N},\label{thetagamma}
\end{equation}
and therefore, from \eqref{theta}, it follows that
\begin{equation}
\displaystyle\lim_j \iota_j = \tilde{\sigma}. \label{iota}
\end{equation}
Arguing by induction, we shall prove that
$$V_p(Du), \ \sqrt{a(x)}V_q(Du) \in B^{\iota_j}_{2, \infty,\text{loc}}(\Omega) \ \quad \forall j \in \mathbb{N}.$$
The case $j=0$ follows from Theorem \ref{thmbes} and our choice of $\iota_0$. Now, let us prove the implication
\begin{equation}
V_p(Du), \ \sqrt{a(x)}V_q(Du) \in B^{\iota_{j-1}}_{2, \infty,\text{loc}}(\Omega)  \Rightarrow V_p(Du), \ \sqrt{a(x)}V_q(Du) \in B^{\iota_j}_{2, \infty,\text{loc}}(\Omega) . \label{ind}
\end{equation}
By virtue of Lemma \ref{besovembed}, the assumptions $V_p(Du) ,  \sqrt{a(x)}V_q(Du)\in B^{\iota_{j-1}}_{2, \infty, \text{loc}}(\Omega)$ imply $$V_p(Du), \ \sqrt{a(x)}V_q(Du) \in L^{\frac{2n}{n-2 \lambda}}(\hat{\mathcal{Q}}_2),$$ for every $0 < \lambda < \iota_{j-1}$ and so, recalling equality \eqref{Vp}, we have that 
$$|Du|^p, \ a(x)|Du|^q \in L^{\frac{n}{n-2 \lambda}}(\hat{\mathcal{Q}}_2).$$
In particular, it follows
$$H(x,Du) \in L^{\frac{n}{n-2 \lambda}}(\hat{\mathcal{Q}}_2),$$
for every $0 < \lambda < \iota_{j-1}$.
Moreover,  the assumption $D \psi \in B^{\gamma}_{2q-p,\infty,\text{loc}}(\Omega)$ and Lemma \ref{besovembed} imply that $D \psi \in L^{\frac{n(2q-p)}{n-\pi (2q-p)}}(\hat{\mathcal{Q}}_2)$, for every $0 < \pi < \gamma$.  Therefore, using H\"{o}lder's inequality in estimate \eqref{1.3} we infer
\begin{align}
&\displaystyle\int_{\mathcal{B}} |\tau_hV_p(Du)|^2+|\tau_h(\sqrt{a(x)}V_q(Du))|^2 \dd x\notag\\
\leq &  C|h|^{\sigma \delta} \displaystyle\int_{\hat{\mathcal{Q}}_2}(1+(H(x,Du))^m+(H(x,D\psi))^m)\dd x \notag\\
&+ C |h|^{\alpha-2 \delta \tilde{p}_1+2 \delta \kappa_1 \lambda} \biggl(  \displaystyle\int_{\hat{\mathcal{Q}}_2}  (1+H(x,Du))^{\frac{n}{n-2 \lambda}}\dd x\biggr)^{\frac{(n-2 \lambda)\kappa_1}{n}}+  C |h|^{\alpha} [D \psi ]^{q_1}_{B^{\gamma}_{2q-p, \infty}(\hat{\mathcal{Q}}_2)} \notag\\
&+ C |h|^{\alpha-2 \delta \tilde{p}_1+(2q-p)\delta\kappa_1 \pi} \biggl(  \displaystyle\int_{\hat{\mathcal{Q}}_2} |D\psi|^{\frac{n(2q-p)}{n-\pi(2q-p)}}\dd x\biggr)^{\frac{(n- \pi (2q-p))\kappa_1}{n}}+C|h|^\alpha \notag\\
\leq &  C|h|^{\sigma \delta} \displaystyle\int_{\hat{\mathcal{Q}}_2}(1+(H(x,Du))^m+(H(x,D\psi))^m)\dd x \notag\\
&+ C |h|^{\alpha-2 \delta \tilde{p}_1+2 \delta \kappa_1 \lambda} \biggl(  \displaystyle\int_{\hat{\mathcal{Q}}_2}  (1+H(x,Du))^{\frac{n}{n-2 \lambda}}\dd x\biggr)^{\frac{(n-2 \lambda)\kappa_1}{n}}+  C |h|^{\alpha} [D \psi ]^{q_1}_{B^{\gamma}_{2q-p, \infty}(\hat{\mathcal{Q}}_2)} \notag\\
&+ C |h|^{\alpha-2 \delta \tilde{p}_1+2 \delta \kappa_1 \lambda} \biggl(  \displaystyle\int_{\hat{\mathcal{Q}}_2} |D\psi|^{\frac{n(2q-p)}{n-\pi(2q-p)}}\dd x\biggr)^{\frac{(n- \pi (2q-p))\kappa_1}{n}}+C|h|^\alpha, \label{diffquo}
\end{align}
for some $\pi \geq \frac{2 \lambda}{2q-p}$, where we used the fact that the radius of the cube $\hat{\mathcal{Q}}_2$ is proportional to $|h|^\delta$.
Therefore, choosing $\delta$ in order to maximize the right hand side of \eqref{diffquo}, namely
$$\delta = \dfrac{\alpha}{\sigma + 2 (\tilde{p}_1-k_1 \lambda)},$$
we have
\begin{align}
&\displaystyle\int_{\mathcal{B}} |\tau_hV_p(Du)|^2+|\tau_h(\sqrt{a(x)}V_q(Du))|^2 \dd x\notag\\
\leq &  C|h|^{\frac{\alpha \sigma}{\sigma + 2 (\tilde{p}_1-k_1 \lambda)}} \biggl\{ \displaystyle\int_{\hat{\mathcal{Q}}_2}(1+(H(x,Du))^m+(H(x,D\psi))^m)\dd x \notag\\
&+   \displaystyle\int_{\hat{\mathcal{Q}}_2}  (1+H(x,Du))^{\frac{n}{n-2 \lambda}}\dd x +  \Vert D \psi \Vert_{B^{\gamma}_{2q-p, \infty}(\hat{\mathcal{Q}}_2)} +1 \biggr\}^{\kappa^*}, 
\end{align}
where $\kappa^* := \kappa^*(n,p,q,\mu,\lambda)$. Thus, again through a covering argument, we deduce that
$$ V_p(Du), \ \sqrt{a(x)}V_q(Du) \in B^{\frac{\alpha \sigma}{2[\sigma + 2 (\tilde{p}_1-k_1 \lambda)]}}_{2,\infty, \text{loc}}(\Omega)= B^{A(\lambda)}_{2,\infty,\text{loc}}(\Omega), \quad \forall \lambda < \iota_{j-1}. $$
We have just proved the following implication
\begin{equation}
V_p(Du), \ \sqrt{a(x)}V_q(Du) \in B^{\iota_{j-1}}_{2,\infty, \text{loc}}(\Omega) \Rightarrow V_p(Du), \ \sqrt{a(x)}V_q(Du) \in B^{t}_{2,\infty, \text{loc}}(\Omega), \label{impbes}
\end{equation}
for all $t < A(\iota_{j-1})$. 

Since $A$ is increasing, it follows from \eqref{thetagamma} that $\theta_j < A(\iota_{j-1})$. Moreover, the definition of $\iota_j$ implies $\iota_j < A(\iota_{j-1})$. Therefore, \eqref{ind} follows from \eqref{impbes}. Besides, from \eqref{thetagamma} and \eqref{iota}, we infer
$$  V_p(Du), \ \sqrt{a(x)}V_q(Du) \in B^{t}_{2,\infty, \text{loc}}(\Omega), \quad \forall t \in ( 0, \tilde{\sigma} ).$$

It is worth noting that the exponent $\tilde{\sigma}$ defined in \eqref{funA} is bigger than $\sigma_\alpha$. Therefore, Theorem \ref{mainthm} improves the higher fractional differentiability result established in Theorem \ref{thmbes}.

\medskip
\textbf{Acknowledgments}
The authors would like to thank Prof. Eleuteri and Prof. Passarelli di Napoli for suggesting the problem and for careful reading.

\end{document}